\documentclass[12pt]{amsart}
\usepackage{epsfig}
\usepackage{pdfsync}
\usepackage{colortbl}
\usepackage[english]{babel}
\newtheorem{thm}{Theorem}[section]
\newtheorem{cor}[thm]{Corollary}
\newtheorem{lem}[thm]{Lemma}
\newtheorem{prop}[thm]{Proposition}

\newtheorem{richiamo}[thm]{}
\theoremstyle{definition}
\newtheorem{defn}[thm]{Definition}
\theoremstyle{remark}
\newtheorem{rem}[thm]{Remark}
\numberwithin{equation}{section}
\newtheorem{ex}[thm]{\bf Example}

\def\L2{L^{2}}

\def\M{\mathcal{M}}
\def\E{\mathcal{E}}

\def\FF{\Bbb{F}}
\def\Z{\Bbb{Z}}

\def\m1{^{-1}}

\def\H{\mathcal{H}}
\def\F{\mathcal{F}}

\def\o{\omega}
\def\O{\Omega}

\begin{document}

\title[]{Amenability and subexponential spectral growth rate of Dirichlet forms on von Neumann algebras}%
\author{Fabio Cipriani, Jean-Luc Sauvageot}%
\address{Dipartimento di Matematica, Politecnico di Milano, piazza Leonardo da Vinci 32, 20133 Milano, Italy.}%
\email{fabio.cipriani@polimi.it}%
\address{Institut de Math\'ematiques, CNRS-–Universit\'e Denis Diderot, F-75205 Paris Cedex 13, France}
\email{jean-luc.sauvageot@imj-prg.fr}
\thanks{}

\date{December 3rd, 2016, revised May 10th 2017, final revision October 8th 2017}
\dedicatory{}
\footnote{This work has been supported by GREFI-GENCO INDAM Italy-CNRS France and MIUR PRIN 2012 Project No 2012TC7588-003}
\begin{abstract}

In this work we apply Noncommutative Potential Theory to characterize (relative) amenability and the (relative) Haagerup Property $(H)$ of von Neumann algebras in terms of the spectral growth of Dirichlet forms. Examples deal with (inclusions of) countable discrete groups and free orthogonal compact quantum groups.
\end{abstract}
\maketitle
\section{Introduction and description of the results.}
Classical results relate the metric properties of conditionally negative definite functions on a countable discrete group $\Gamma$ to its approximation properties. For example, there exists a {\it proper}, conditionally negative definite function $\ell$ on $\Gamma$ if and only if there exists a sequence $\varphi_n\in c_0(\Gamma)$ of normalized, positive definite functions, vanishing at infinity and converging pointwise to the constant function $1$.
\par\noindent
In a celebrated work [Haa2], U. Haagerup proved that the length function of a free group $\mathbb{F}_n$ with $n\in\{2,\cdots ,\infty\}$ generators is negative definite, thus establishing for free groups the above approximation property. Since then the property is referred to as Haagerup Approximation Property (H) or Gromov a-T-menability (see [CCJJV]).\par\noindent
In addition, if for a conditionally negative definite function $\ell$ on a countable discrete group $\Gamma$, the series $\sum_{g\in\Gamma}e^{-t\ell (g)}$ converges for all $t>0$, then there exists a sequence $\varphi_n\in l^2(\Gamma)$ of normalized, positive definite functions, converging pointwise to the constant function $1$ ([GK Thm 5.3]). This latter property is just one of the several equivalent appearances of {\it amenability}, a property introduced by J. von Neumann in 1929 [vN] in order to explain the Banach-Tarski paradox in Euclidean spaces $\mathbb{R}^n$ exactly when $n\ge 3$.
\vskip0.1truecm\noindent
In this note we are going to discuss extensions of the above results concerning amenability for $\sigma$-finite von Neumann algebras $N$.
\vskip0.1truecm\noindent
The direction along which we are going to look for substitutes of the above summability condition related to amenability, is that of Noncommutative Potential Theory.
\vskip0.1truecm\noindent
This is suggested by a recent result by Caspers-Skalski [CaSk] asserting that $N$ has the (suitably formulated) Haagerup Approximation Property (H) if and only if there exists a Dirichlet form $(\E,\F)$ on the standard Hilbert space $L^2(N)$, having {\it discrete spectrum}.
\par\noindent
The link between the properness condition for a conditionally negative definite function $\ell$ on a countable discrete groups $\Gamma$ and the generalized one on von Neumann algebras, relies on the fact that, when the von Neumann algebra $N =L(\Gamma)$ is the one generated by the left regular representation of $\Gamma$, the quadratic form $\E_\ell[a]=\sum_{g\in\Gamma}\ell (g)|a(g)|^2$ on the standard space $L^2(L(\Gamma),\tau)\simeq l^2(\Gamma)$ is a Dirichlet form if and only if the function $\ell$ is conditionally negative definite and its spectrum is discrete if and only if $\ell$ is proper. Moreover, on a countable, finitely generated, discrete group $\Gamma$ with polynomial growth, there exist a conditionally negative definite functions $\ell$, having polynomial growth and growth dimensions arbitrarily close to the homogeneous dimension of $\Gamma$ (see [CS5]).
\vskip0.1truecm\noindent
This point of view thus suggests that a condition providing amenability of a von Neumann algebra with faithful normal state $(N,\o)$ could be the {\it subexponential spectral growth} of a Dirichlet form $(\E,\F)$ on the standard space $L^2(N,\o)$, i.e. the discreteness of the spectrum of $(\E,\F)$ and the summability of the series $\sum_{k\ge 0}e^{-t\lambda _k}$ for all $t>0$, where $\lambda_0\, ,\lambda_1\, ,\dots$ are the eigenvalues of $(\E,\F)$.
\vskip0.1truecm\noindent
The second fundamental fact that will allow to use Dirichlet forms to investigate the amenability of a von Neumann algebra, is the possibility to express this property in terms of Connes' {\it correspondences}: $N$ is amenable if and only if the identity or standard $N$-$N$-correspondence $L^2(N)$ is weakly contained in the coarse or Hilbert-Schmidt $N$-$N$-correspondence $L^2(N)\otimes L^2(N)$ (see [Po1]).
\vskip0.1truecm\noindent
In the second part of the work we provide a condition guaranteeing the {\it relative amenability} of an inclusion $B\subseteq N$ of finite von Neumann algebras introduced by Popa [Po1,2],
in terms of the existence of a Dirichlet form $(\E,\F)$ on $L^2(N)$ having {\it relative subexponential spectral growth}. Also this result is based on the possibility to express the relative amenability of a von Neumann algebra $N$ with respect to a subalgebra $B\subseteq N$ in terms of the weak containment of the identity correpondence $L^2(N)$ in the relative tensor product correspondence $L^2(N)\otimes_B L^2(N)$ introduced by Sauvageot [S1], [Po2].
\vskip0.1truecm\noindent
Using a suitable Dirichlet form constructed in [CFK], whose construction uses tools developed by M. Brannan in [Bra], we apply the above result to prove amenability of the von Neumann algebra of the free orthogonal quantum group $O^+_2$ and Haagerup Property (H) of the free orthogonal quantum groups $O^+_N$ for $N\ge 3$ (see also the recent [DFSW]), results firstly obtained by M. Brannan [Bra].
\par\noindent
A detailed discussion of the relative Haagerup Property (H) for inclusions of countable discrete groups in terms of conditionally negative definite functions is presented.
\vskip0.2truecm\noindent
The paper is organized as follows: in Section 2 we provide the necessary tools on noncommutative potential theory on von Neumann algebra as Dirichlet forms, Markovian semigroups and resolvents.
\par\noindent
In Section 3 we first recall some equivalent constructions of the coarse or Hilbert-Schmidt correspondence of a von Neumann algebra $N$ and some connections between the modular theories of $N$, of its opposite $N^o$, and of their spatial tensor product $N\overline{\otimes}N^o$. Then we introduce the spectral growth rate of a Dirichlet form and we prove the first main result of the work about the amenability of von Neumann algebra admitting a Dirichlet form with subexponential spectral growth rate. This part terminates with an application to the amenability of countable discrete groups and with a partially alternative approach to the proof of a result of M. Brannan [Bra] about the amenability of the free orthogonal quantum group $O_2^+$.
\par\noindent
Section 4 starts recalling some fundamental tool of the basic construction $\langle N,B\rangle$ for inclusions $B\subseteq N$ of finite von Neumann algebras, needed to prove the second main result of the work concerning the amenability of $N$ with respect to its subalgebra $B$. To formulate the criterion, we introduce the spectral growth rate of a $B$-invariant Dirichlet form on the standard space $L^2(N)$ relatively to the subalgebra $B$, using the compact ideal space $\mathcal{J}(\langle N,B\rangle)$ of $\langle N,B\rangle$ ( cf. [PO1,2]). The section terminates discussing relative amenability for two natural subalgebras $B_{\rm min}\subseteq N$ and $B_{\rm max}\subseteq N$ associated to any Dirichlet form.
\par\noindent
In Section 5 we extend the spectral characterization of the Haagerup Property (H) of von Neumann algebras with countable decomposable center due to M. Caspers and A. Skalski [CaSk] to the Relative Haageruup Property (H) for inclusions of finite von Neumann algebras $B\subseteq N$ formulated by S. Popa [Po 1,2].
\par\noindent
In Section 6 we discuss the relative Haagerup Property (H) for inclusions $H<G$ of countable discrete groups in terms of the existence of an $H$-invariant conditionally negative definite function on $G$ which is proper on the homogeneous space $G/H$ and in terms of quasi-normality of $H$ in $G$.
\vskip0.2truecm
The content of the present work has been the subject of talks given in Rome II (March 2015), Paris (GREFI-GENCO April 2015), Berkeley (UC Seminars September 2015), Krakov (September 2015), Varese (May 2016).
\section{Dirichlet forms on $\sigma$-finite von Neumann algebras}
Recall that a von Neumann algebra $N$ is $\sigma$-finite, or countably decomposable, if any collection of mutually orthogonal projections is at most countable and that this property is equivalent to the existence of a normal, faithful state. This is the case, for example, if $N$ acts faithfully on a separable Hilbert space.
\vskip0.2truecm\noindent
Let us consider on a $\sigma$-finite von Neumann algebra $N$ a fixed faithful, normal state $\o\in N_{*+}$. Let us denote by $(N, L^2(N,\o),L^2_+(N,\o), J_\o)$ the standard form of $N$ and by $\xi_\o\in L^2_+(N,\o)$ the cyclic vector representing the state (see [Haa1]).
\par\noindent
For a real vector $\xi=J_\o\xi\in L^2(N,\o)$, let us denote by $\xi\wedge\xi_\o$ the Hilbert projection of the vector $\xi$ onto the closed and convex set $C_\o :=\{\eta\in L^2(N,\o):\eta =J_\o\eta ,\,\,\xi_\o -\eta\in L^2_+(N,\o)\}$.
\vskip0.2truecm\noindent
We recall here the definition of Dirichlet form and Markovian semigroup (see [C1]) on a generic standard form of a $\sigma$-finite von Neumann algebra. For a definition particularized to the Haagerup standard form see [GL1].
\begin{defn}[Dirichlet forms on $\sigma$-finite von Neumann algebras]
A densely defined, nonnegative and lower semicontinuous quadratic form $\E:L^2(N,\o )\rightarrow [\,0,+\infty]$ is said to be:
\item{i)} {\it real} if
\begin{equation}
\E [J_\o(\xi)]=\E [\xi]\qquad  \xi\in L^2(N,\o)\, ;
\end{equation}
\item{ii)} a {\it Dirichlet form} if it is real and {\it Markovian} in the sense that
\begin{equation}
\E [\xi\wedge \xi_\o ]\le\E [\xi]\qquad  \xi=J_\o\xi\in L^2(N,\o)\, ;
\end{equation}
\item{iii)} a {\it completely Dirichlet form} if all the canonical extensions $\E_n$ to $L^2 (\mathbb{M}_n (N), \o\otimes {\rm tr\,}_n )$
\begin{equation}
\E_n [[\xi_{i,j}]_{i,j=1}^n] :=\sum_{i,j=1}^n \E[\xi_{i,j}]\qquad
[\xi_{i,j}]_{i,j=1}^n\in L^2 (\mathbb{M}_n (N), \o\otimes {\rm tr\,}_n )\, ,
\end{equation}
are Dirichlet forms.
\end{defn}
\noindent
By the self-polarity of the standard cone $L^2_+(N,\o)$, any real vector $\xi=J_\o\xi\in L^2(N,\o)$ decomposes uniquely as a difference $\xi=\xi_+ -\xi_-$ of two positive, orthogonal vectors $\xi_\pm\in L^2_+(\M,\o)$ (the positive part $\xi_+$ being just the Hilbert projection of $\xi$ onto the positive cone). The modulus of $\xi$ is then defined as the sum of the positive and negative parts $|\xi|:=\xi_+  + \xi_-$.
\par\noindent
Notice that, in general, the contraction property
\[
\E [\,|\xi|\, ]\le\E [\xi]\qquad  \xi=J_\o\xi\in L^2(A,\o)
\]
is a consequence of Markovianity and that it is actually equivalent to it when $\E [\xi_\o]=0$.
\vskip0.2truecm\noindent
The domain of the Dirichlet form is defined as the (dense) subspace of $L^2(N,\o)$ where the quadratic form is finite: $\F:=\{\xi\in L^2(N,\o): \E[\xi]<+\infty\}$. We will denote by $(L, D(L))$ the densely defined, self-adjoint, nonnegative operator on $L^2(A,\tau)$ associated with the closed quadratic form $(\E ,\F)$
\[
\F=D(\sqrt L)\qquad {\rm and}\qquad \E[\xi]=\|\sqrt{L}\xi\|^2\qquad \xi\in D(\sqrt L)=\F\, .
\]

\begin{defn}[Markovian semigroups on standard forms of von Neumann algebras]~
\par\noindent
a) A bounded operator $T$ on $L^2(N,\o)$ is said to be
\item{i)} {\it real} if it commutes with the modular conjugation: $T J_\o =J_\o T$,
\item{ii)} {\it positive} if it leaves globally invariant the positive cone: $T(L^2_+(N,\o))\subseteq L^2_+(N,\o)$,
\item{iii)} {\it Markovian} if it is real and it leaves globally invariant the closed, convex set $C_\o$:
\[
T(C_\o)\subseteq C_\o\, ,
\]
\item{iv)} {\it completely positive, resp. completely Markovian}, if it is real and all of its matrix amplifications $T^{(n)}$ to $L^2 (\mathbb{M}_n (N), \o\otimes {\rm tr\,}_n )\simeq L^2(N,\tau)\otimes L^2(\mathbb{M}_n (\mathbb{C}),{\rm tr\,}_n)$ defined by
\[
T^{(n)} [[\xi_{i,j}]_{i,j=1}^n] :=\sum_{i,j=1}^n [T\xi_{i,j}]_{i,j=1}^n\qquad
[\xi_{i,j}]_{i,j=1}^n\in L^2 (\mathbb{M}_n (N), \o\otimes {\rm tr\,}_n )\, ,
\]
are positive, resp. Markovian;
\par\noindent
b) A strongly continuous, uniformly bounded, self-adjoint semigroup $\{T_t:t>0\}$ on $L^2(N,\o)$ is said to be real (resp. positive, Markovian, completely positive, completely Markovian) if the operators $T_t$ are real (resp. positive, Markovian, completely positive, completely Markovian) for all $t>0$.
\end{defn}
\noindent
In literature, property in item iii) above is sometime termed {\it submarkovian}, while {\it markovian} is meant positivity preserving {\it and unital}. Our choice is only dictated by a willing of simplicity.
\vskip0.2truecm\noindent
Notice that $T$ is (completely) Markovian iff it is (completely) positive and $T\xi_\o\le\xi_\o$.
\vskip0.2truecm\noindent
Notice that if $N$ is abelian, then positive (resp. Markovian) operators are automatically completely positive (resp. completely Markovian).
\vskip0.2truecm\noindent
Dirichlet forms are in one-to-one correspondence with Markovian semigroups (see [C1]) through the relations
\[
T_t =e^{-tL}\, \qquad t\geq 0\]
where $(L,D(L))$ is the self-adjoint operator associated to the quadratic form $(\E,\F)$.
\vskip0.2truecm
Dirichlet forms and Markovian semigroups are also in correspondence with a class of semigroups on the von Neumann algebra. To state this fundamental relation, let us consider the {\it symmetric embedding} $i_\omega$ determined by the cyclic vector $\xi_\omega$
\[
i_\o:N\rightarrow L^2(N,\o)\qquad i_\o (x):=\Delta_\o^{\frac{1}{4}}x\xi_\o\qquad x\in N\, .
\]
Here, $\Delta_\o$ is the modular operator associated with the faithful normal state $\omega$ (see [T]).  We will denote by $\{\sigma^\o_t:t\in\mathbb{R}\}$ the modular automorphisms group associated to $\o$  and by $N_{\sigma^{\o}}\subseteq N$ the subalgebra of elements which are analytic with respect to it. Then (see [C1]) (completely) Dirichlet forms $(\E,\F)$ and (completely) Markovian semigroups $\{T_t:t>0\}$ on $L^2(N,\o)$ are in one-to-one correspondence with those weakly$^*$-continuous, (completely) positive and contractive semigroups $\{S_t:t>0\}$ on the von Neumann algebra $N$ which are {\it modular $\o$-symmetric} in the sense that
\begin{equation}\label{symmetry}
\o (S_t (x)\sigma^{\o}_{-i/2}(y))=\o (\sigma^{\o}_{-i/2}(x)S_t (y))\qquad x,y\in N_{\sigma^{\o}}\, ,\quad t>0\, ,
\end{equation}
through the relation
\[
\quad i_\o (S_t (x))=T_t (i_\o (x))\qquad x\in N\, ,\quad t>0\, .
\]
Relation (\ref{symmetry}) is called {\it modular symmetry} and it is equivalent to
\begin{equation}\label{symmetry2}
(J_\o y\xi_\o|S_t (x)\xi_\o)=(J_\o S_t(y)\xi_\o |x\xi_\o)\qquad x,y\in N\, ,\quad t>0\, .
\end{equation}

\begin{rem}
In case $\o$ is a trace, the symmetric embedding reduces to $i_\o(x)=x\xi_\o$ while the modular symmetry simplifies to $\o (S_t (x)y)=\o (xS_t (y))$ for $x,y\in N$ and $t>0$.
\end{rem}
\vskip0.2truecm
\centerline{{\it To shorten notations, in the forthcoming part of the paper}}
\centerline{{\it "Dirichlet form" will always mean "completely Dirichlet form" and}}
\centerline{{\it "Markovian semigroup" will always mean "completely Markovian semigroup".}}
\vskip0.2truecm\noindent
Whenever no confusion can arise, the modular conjugation $J_\o$ will be sometime denoted $J$.

\subsection{Examples of Dirichlet forms}
Instances of the notions introduced above may be found in various frameworks. We just recall here some examples of different origins. One may consult the fundamental works [BeDe], [FOT] for the commutative case and [C2], [C3] for surveys in the noncommutative setting.\par\noindent
a) The archetypical Dirichlet form on the Euclidean space $\mathbb{R}^n$ or, more generally, on any Riemannian manifold $V$, endowed with its Riemannian measure $m$, is the Dirichlet integral
\[
\E[a]=\int_V |\nabla a|^2\, dm\qquad a\in L^2 (V,m)\, .
\]
In this case the trace on $L^\infty (V,m)$ is given by the integral with respect to the measure $m$ and the form domain is the Sobolev space $H^1 (V)\subset L^2 (V,m)$. The associated Markovian semigroup is the familiar heat semigroup of the Riemannian manifold. Interesting variations of the above Dirichlet integral are the Dirichlet forms of type
\[
\E[a]:=\int_{\mathbb{R}^n} |\nabla a|^2\, d\mu\qquad a\in L^2 (\mathbb{R}^n,\mu)\, ,
\]
that for suitable choices of positive Radon measures $\mu$, are ground state representations of Hamiltonian operators in Quantum Mechanics.
\par\noindent
b) Dirichlet forms are a fundamental tool to introduce differential calculus and study Markovian stochastic processes on fractal sets (see [Ki], [CS3], [CGIS 1,2]).
\par\noindent
c) On a countable discrete group $\Gamma$, any conditionally negative definite function $\ell$ gives rise to a Dirichlet form
\[
\E_\ell [\xi] := \sum_{s\in\Gamma} |\xi(s)|^2 \ell(s)\, ,
\]
on the Hilbert space $l^2 (\Gamma)$, considered as the standard Hilbert space of the left von Neumann algebra $L(\Gamma)$ generated by the left regular representation of $\Gamma$ (see [CS1], [C2]). The associated Markovian semigroup is simply given by the multiplication operator
\[
T_t (a)(s)=e^{-t\ell(s)}a(s)\qquad t>0\,,\quad s\in G\,,\quad a\in l^2(\Gamma)\, .
\]
d) On noncommutative tori $A_\theta$,  $\theta\in [0,1]$ (see [Co2]), which are C$^*$-algebras generated by two unitaries $u$ and $v$, satisfying the relation
\[
vu=e^{2i\pi\theta}uv\, ,
\]
the {\it heat semigroup} $\{T_t :t\ge 0\}$ defined by
\[
T_t(u^nv^m)= e^{-t(n^2+m^2)}u^nv^m \qquad (n,m)\in \mathbb{Z}^2\, ,
\]
is a $\tau$-symmetric Markovian semigroup on the von Neumann algebra $N_\theta$ generated by the G.N.S. representation of the faithful, tracial state $\tau : A_\theta \to\mathbb{C}$ characterized by
\[
\tau(u^n v^m)=\delta_{n,0}\delta_{m,0} \qquad n,m\in \mathbb{Z}\, .
\]
e) There exists a general interplay between Dirichlet forms and differential calculus on tracial C$^*$-algebras $(A,\tau)$ (see [S 2,3], [CS1]) and this provides a source of Dirichlet forms on von Neumann algebras (generated by $A$ in the G.N.S. representation of the trace). In fact, denoting by $N$ the von Neumann algebra generated by the G.N.S. representation of the trace, if  $(\partial, D(\partial))$ is a densely defined closable derivation from $L^2(N,\tau)$ to Hilbert $A$-bimodule $\H$, then the closure of the quadratic form
\[
\E[a]:=\|\partial a\|^2_\H\qquad a\in \F:=D(\partial)
\]
is a Dirichlet form on $L^2(N,\tau)$. Viceversa, any Dirichlet form on $L^2(N,\tau)$ whose domain is dense in $A$  arises in this way from an essentially unique derivation on $A$ canonically associated with it (see [CS1]).
Examples of this differential calculus can be found in all the situations illustrated above as well as in the geometric framework of Riemannian foliations (see [S4]) and also in the framework of Voiculescu's Free Probability theory (see [V1]). There, the Dirichlet form associated to Voiculescu's derivation presents several aspects connected to  Noncommutative Hilbert Transform, Free Fischer Information and Free Entropy.

\section{Amenability of $\sigma$-finite von Neumann algebras}
In this section we relate a certain characteristic of the spectrum of a Dirichlet form to the amenability of the von Neumann algebra. Recall that a von Neumann algebra $N$ is said to be {\it amenable} if, for every normal dual Banach $N$-bimodule $X$, the derivations $\delta:N\to X$ are all inner, i.e. they have the form
\[
\delta (x)=x\xi-\xi x\qquad x\in N
\]
form some vector $\xi\in X$. It is a remarkable fact, and the byproduct of a tour de force, that this property is equivalent to several others of apparently completely different nature, such as {\it hyperfiniteness, injectivity, semi-discreteness, Schwartz property P, Tomiyama property E}. We refer to [Co2 Ch. V] for a review on these connections. Among the main examples of amenable von Neumann algebras, we recall: the von Neumann algebra of a locally compact amenable group, the crossed product of an abelian von Neumann algebra by an amenable locally compact group, the commutant von Neumann algebra of any continuous unitary representation of a connected locally compact group, the von Neumann algebra generated by any representation of a nuclear C$^*$-algebra.
\subsection{Standard form of the spatial tensor product of von Neumann algebras}
Here we summarize some well known properties of the standard form of the spatial tensor product of two von Neumann algebras in terms of Hilbert-Schmidt operators (details may be found in [T]), mainly with the intention to make precise, in the next section, some properties of the symmetric embedding of a product state. More precisely we shall use the following facts:
\begin{richiamo}
let $N\subseteq \mathcal{B(H)}$ be a von Neumann algebra. A vector $\xi\in H$ is cyclic for the commutant $N'$ if and only if it is separating for $N$;
\end{richiamo}

\begin{richiamo}
let $N_k\subseteq \mathcal{B}(H_k)$ $k=1,2$ be von Neumann algebras. If the vectors $\xi_k\in H_k$ , $k=1,2$ are cyclic for $N_k$, then the vector $\xi_1\otimes \xi_2\in H_1\otimes H_2$ is cyclic for the spatial tensor product $N_1\overline{\otimes} N_2$;
\end{richiamo}
\begin{richiamo}
let $N_k$ $k=1,2$ be von Neumann algebras and $L^2(N_k)$ their standard forms. If the vectors $\xi_k\in L^2_+ (N_k)$ $k=1,2$ are cyclic for $N_k$ (hence separating) then the vector $\xi_1\otimes \xi_2\in L^2 (N_1)\otimes L^2 (N_2)$ is cyclic and separating for the spatial tensor product $N_1\overline{\otimes} N_2$.
\end{richiamo}
\subsection{Symmetric embedding of tensor product of von Neumann algebras}
Here we recall the definition and a property of the symmetric embedding of a von Neumann algebra in its standard Hilbert space. Let $N$ be a $\sigma$-finite von Neumann algebra and $\o\in N_{*,+}$ a faithful, normal state.
\par\noindent
In the standard form $(N, L^2(N,\o), L^2_+ (N,\o))$, we denote by $\xi_\o\in L^2_+(N,\o)$ the cyclic vector representing the state $\o$ and by $J_\o$ and $\Delta_\o$ its modular conjugation and modular operator, respectively.
\par\noindent
The symmetric embedding $i_\o:N\to L^2(N,\o)$, defined by $i_\o (x):=\Delta_\o^{\frac{1}{4}}x\xi_\o$ for $x\in N$, is a completely positive contraction with dense range, which is also continuous between the weak$^*$-topology of $N$ and the weak topology of $L^2(N,\o)$. It is also an order isomorphism of completely ordered sets between $\{x=x^*\in N:0\le x\le 1_N\}$ and $\{\xi=J_\o\xi\in L^2(N,\o): 0\le\xi\le\xi_\o)\}$ (see [Ara], [Co1], [Haa1] and [BR]). We shall make use of the following properties:
\begin{richiamo}
Let $N_k$ $k=1,2$ be von Neumann algebras and $L^2(N_k)$ their standard forms. Consider the cyclic (hence separating) vectors $\xi_k\in L^2 (N_k)$ $k=1,2$ and the cyclic and separating vector $\xi_1\otimes \xi_2\in H_1\otimes H_2$ for the spatial tensor product $N_1\overline{\otimes} N_2$.
\par\noindent
Let $J_k, \Delta_k$ be the modular conjugation and the modular operator associated to $\xi_k\in H_k$ $k=1,2$ and $J_{\xi_1\otimes\xi_2}, \Delta_{\xi_1\otimes\xi_2}$ be the modular conjugation and the modular operator associated to $\xi_1\otimes \xi_2$. Then the following identifications hold true
\begin{itemize}
\item $J_{\xi_1\otimes\xi_2}=J_{\xi_1} \otimes J_{\xi_2}$;
\item $N_1\xi_1 \odot N_2\xi_2\subseteq H_1\otimes H_2$ is a core for the closed operator $\Delta^{\frac{1}{2}}_{\xi_1\otimes\xi_2}$;
\item $\Delta^{\frac{1}{2}}_{\xi_1\otimes\xi_2}(\eta_1\otimes \eta_2) = \Delta^{\frac{1}{2}}_{\xi_1}(\eta_1)\otimes \Delta^{\frac{1}{2}}_{\xi_2}(\eta_2)$ for $\eta_k\in N_k\xi_k$ and $k=1,2$.
\end{itemize}
\end{richiamo}
\par\noindent
We will denote by $N^\circ$ the opposite algebra of $N$: it coincides with $N$ as a vector space but the product is taken in the reverse order $x^\circ y^\circ :=(yx)^\circ$ for $x^\circ , y^\circ\in N^\circ$. As customary, we adopt the convention that elements $y\in N$, when regarded as elements of the opposite algebra are denoted by $y^\circ\in N^\circ$.
\par\noindent
A linear functional $\o$ on $N$, when considered as a linear functional on the opposite algebra $N^\circ$ is denoted by $\o^\circ$ and called the {\it opposite} of $\o$. As $N$ and $N^\circ$ share the same positive cone, if $\o$ is positive on $N$ so is $\o^\circ$ on $N^\circ$ and if $\o$ is normal so does its opposite.
\vskip0.2truecm\noindent
By the properties of standard forms of von Neumann algebras, it follows that for the standard form $(N^\circ, L^2(N^\circ,\o^\circ), L^2_+ (N^\circ,\o^\circ))$ of $N^\circ$ one has the following identifications
\[
L^2(N^\circ,\o^\circ)=L^2(N,\o)\, ,\quad L^2_+(N^\circ,\o^\circ)=L^2_+(N,\o)\, ,\quad J_\o = J_{\o^\circ}\, ,\quad\Delta_{\o^\circ}=\Delta_\o^{-1}\, , \quad\xi_{\o^\circ}=\xi_\o\, .
\]
Using the isomorphism between $N^\circ$ and the commutant $N^\prime$, given by $N^\circ\ni y^\circ\rightarrow J_\o y^* J_\o\in N^\prime$, we can regard $L^2(N,\o)$ not only as a left $N$-module but also as a left $N^\circ$-module, hence as a right $N$-module and finally as a $N$-$N$-bimodule
\[
y^\circ\xi := J_\o y^*J_\o\xi\, ,\quad \xi y :=J_\o y^*J_\o\xi\, ,\quad x\xi y :=xJ_\o y^*J_\o\xi\,\qquad x, y\in N\, ,\xi\in L^2 (N,\o)\, .
\]
\par\noindent
The symmetric embeddings associated to $\o$ and $\o^\circ$ are related by
\[
i_{\o^\circ}(y^\circ)=\Delta^{\frac{1}{4}}_{\o^\circ} (\xi_\o y)=\Delta^{\frac{1}{4}}_{\o^\circ} J_\o y^*J_\o\xi_\o=\Delta^{-\frac{1}{4}}_\o \Delta^{\frac{1}{2}}_\o y\xi_\o=\Delta^{\frac{1}{4}}_\o y\xi_\o=i_\o (y)\, .
\]
\[
J_\o (i_\o (y^*))= J_\o \Delta^{\frac{1}{4}}_\o (y^*\xi_\o)=J_\o \Delta^{\frac{1}{4}}_\o J_\o \Delta^{\frac{1}{2}}_\o(y\xi_\o)=\Delta^{\frac{1}{4}}_\o (y\xi_\o)=i_\o (y)=i_{\o^\circ}(y^\circ)\, .
\]

\subsection{Coarse correspondence}
Recall that a Hilbert-Schmidt operator $T$ is a bounded operator on $L^2(N,\o)$ such that ${\rm Trace}_{L^2(N,\o)} (T^* T)<+\infty$. It may be represented as
\[
T\xi:=\sum_{k=0}^\infty \mu_k (\eta_k|\xi)\xi_k\qquad \xi\in L^2(N,\o)
\]
in terms of suitable orthonormal systems $\{\eta_k :k\in \mathbb{N}\}\, ,\,\,\{\xi_k :k\in \mathbb{N}\}\subset L^2(N,\o)$ and a sequence $\{\mu_k :k\in\mathbb{N}\}\subset\mathbb{C}$ such that $\sum_{k=0}^\infty |\mu_k|^2<+\infty$. The set of Hilbert-Schmidt operators $HS(L^2(N,\o))$ is a Hilbert space under the scalar product $(T_1 |T_2):={\rm Trace}_{L^2(N,\o)} (T_1^* T_2)$.

\begin{lem}
The binormal representations $\pi^1_{\rm co}, \pi^2_{\rm co}, \pi^3_{\rm co}$ of $N\otimes_{\rm max}N^\circ$, characterized by
\[
\begin{split}
&\pi^1_{\rm co}:N\otimes_{\rm max}N^\circ\rightarrow \mathcal{B}({\rm HS\, }(L^2(N,\tau))) \\
\pi^1_{\rm co} (x\otimes y^o)(T)&:=xTy\qquad x, y\in N\, ,\quad T\in {\rm HS\, }(L^2(N,\tau))\, , \\
\end{split}
\]
\[
\begin{split}
&\pi^2_{\rm co}:N\otimes_{\rm max}N^\circ\rightarrow \mathcal{B}(L^2(N,\tau)\otimes \overline{L^2(N,\tau)}) \\
\pi^2_{\rm co} (x\otimes y^o)(\xi\otimes\overline{\eta})&:=x\xi\otimes\overline{\eta y}\qquad x, y\in N\, ,\quad \xi, \eta\in L^2(N,\tau)\, \\
\end{split}
\]
\[
\begin{split}
&\pi^3_{\rm co}:N\otimes_{\rm max}N^\circ\rightarrow \mathcal{B}(L^2(N,\tau)\otimes L^2(N,\tau)) \\
\pi^3_{\rm co} (x\otimes y^o)(\xi\otimes\eta)&:=x\xi\otimes\eta y\qquad x, y\in N\, ,\quad \xi, \eta\in L^2(N,\tau)\, , \\
\end{split}
\]
are unitarely equivalent by
\[
\begin{split}
&U:L^2(N,\tau)\otimes \overline{ L^2(N,\tau)}\rightarrow L^2(N,\tau)\otimes {L^2(N,\tau)}\qquad U(\xi\otimes\overline{\eta}):= \xi\otimes J_\o{}\eta \\
&V:L^2(N,\tau)\otimes \overline{L^2(N,\tau)}\rightarrow {\rm HS}(L^2(N,\tau))\qquad  V(\xi\otimes \overline{\eta})(\zeta):=(\eta|\zeta)\xi\, .\\
\end{split}
\]
They give rise by weak closure
\[
(\pi^3_{\rm co}(N\otimes_{\rm max}N^\circ))'' = N\overline{\otimes} N^\circ
\]
of the spatial tensor product of $N$ by its opposite $N^\circ$.
\end{lem}

\begin{lem}
The normal extension of the coarse representation $\pi_{\rm co}$ of the ${\rm C}^*$-algebra $N\otimes_{\rm max}N^\circ$ to the von Neumann tensor product $N\overline{\otimes} N^\circ$ is the standard representation of $N\overline{\otimes} N^\circ$ (and it will still denoted by the same symbol).
\par\noindent
The standard positive cone in the various equivalent representations is determined as
\begin{itemize}
  \item ${\rm HS}(L^2(N,\o))_+$, the set of all nonnegative Hilbert-Schmidt operators on $L^2(N,\o)$;
  \item $(L^2(N,\tau)\otimes \overline{L^2(N,\o)})_+$, generated by the vectors $\xi\otimes\overline{\xi}$ with $\xi\in L^2(N,\o)$;
  \item $(L^2(N,\tau)\otimes L^2(N,\o))_+$, generated by the vectors $\xi\otimes J_\o{}\xi$ with $\xi\in L^2(N,\o)$.
\end{itemize}
The standard Hilbert space and the positive cone of $N\overline{\otimes}N^\circ$ will be denoted also by
\[
L^2(N\overline{\otimes} N^\circ,\o\otimes\o^\circ)\, ,\qquad L^2_+(N\overline{\otimes} N^\circ,\o\otimes\o^\circ)\, .
\]
\end{lem}

\begin{lem}\label{cp}
Let $T:L^2(N,\o)\rightarrow L^2(N,\o)$ be a bounded operator and consider on the involutive algebra $N\odot N^\circ$, the linear functional determined by
\[
\Theta_T:N\odot N^\circ\rightarrow\mathbb{C}\qquad\Theta_T (x\otimes y^\circ):= (i_\o (y^*)|T i_\o (x))\qquad x\otimes y^\circ\in N\odot N^\circ\, .
\]
Then $\Theta_T$ is a positive linear functional on $N\odot N^\circ$ if and only if $T$ is completely positive  (c.f. Definition 2.2 iv)).
\end{lem}
\begin{proof}
i) The positive cone of $N\odot N^\circ$ is generated by elements of type $\nu^*\nu = \sum_{j,k=1}^n x_j^*x_k\otimes (y_k y_j^*)^\circ$ where $\nu =\sum_{k=1}^n x_k\otimes y_k^\circ \in N\odot N^\circ$. The result then follows by the identity
\[
\Theta_T (\nu^*\nu)=\sum_{j,k=1}^n \Theta_T (x_j^*x_k\otimes (y_k y_j^*)^\circ)=\sum_{j,k=1}^n (i_\o (y_j y_k^*)|T i_\o (x_j^*x_k))\, ,
\]
the completely positivity of the symmetric embedding $i_\o :N\rightarrow L^2(N,\tau)$ and the positivity of $[x_j^* x_k]_{j,k=1}^n$ and $[y_j y_k^*]_{j,k=1}^n$ in $\mathbb{M}_n (N)$.\par\noindent
\end{proof}
\begin{lem}
Let $T:L^2(N,\o)\rightarrow L^2(N,\o)$ be a completely positive operator and consider the positive linear functional $\Theta_T$ on $N\odot N^\circ$. Then, among the properties
\vskip0.2truecm\noindent
a) $\Theta_T$ is a state on $N\odot N^\circ$\par\noindent
b) $T$ is a contraction\par\noindent
c) $T\xi_\o =\xi_\o$
\vskip0.2truecm\noindent
we have that the following relations
\vskip0.2truecm\noindent
i) a) and b) imply c) and $\|T\|=1$\par\noindent
ii) c) implies a) and b).
\end{lem}
\begin{proof}
i) By a) and b) we have $1=\Theta_T (1_N\otimes 1_{N^\circ})=(\xi_\o |T\xi_\o)\le \|\xi_\o\|\cdot\|T\xi_\o\|\le \|\xi_\o\|^2 \cdot \|T\|=1$ that implies $\|T\|=\|T\xi_\o\|=1$ and $(\xi_\o |T\xi_\o)=\|\xi_\o\|\cdot\|T\xi_\o\|$ which provide $T\xi_\o =\xi_\o$. ii) The proof that c) implies a) is immediate while the proof that c) implies b) can be found in [C1].
\end{proof}

\subsection{Spectral growth rate}

In the following definition, the notion of growth rate of a finitely generated, countable discrete group is extended to $\sigma$-finite von Neumann algebras having the Haagerup Property (H), i.e. von Neumann algebras admitting Dirichlet forms with discrete spectrum. The idea for this generalization results from [CS5] (see discussion in Example 3.11 below).
\begin{defn}({\bf Spectral growth rate of Dirichlet forms}).
Let $(N,\o)$ be a $\sigma$-finite, von Neumann algebra with a fixed faithful, normal state on it. To avoid trivialities we assume $N$ to be {\it infinite dimensional} .\par\noindent
Let $(\E,\F)$ be a Dirichlet form on $L^2(N,\o)$ and let $(L,D(L))$ be the associated nonnegative, self-adjoint operator. Assume that its spectrum  $\sigma(L)=\{\lambda_k\ge 0: k\in\mathbb{N}\}$ is {\it discrete}, i.e. its points are isolated eigenvalues of finite multiplicity (repeated in non decreasing order according to their multiplicities).\par\noindent
Then let us set
\[
\Lambda_n :=\{k\in\mathbb{N}: \lambda_k\in [0,n]\}\, ,\qquad \beta_n:=\sharp (\Lambda_n)\, ,\qquad n\in\mathbb{N}
\]
and define the {\it spectral growth rate} of $(\E,\F)$ as
\[
\Omega (\E,\F):=\limsup_{n\in\mathbb{N}} \sqrt[n]{\beta_n}\, .
\]
The Dirichlet form $(\E,\F)$ is said to have
\vskip0.2truecm
\begin{itemize}
  \item {\it exponential growth} if $(\E,\F)$ has discrete spectrum and $\Omega (\E,\F)>1$
  \item {\it subexponential growth} if $(\E,\F)$ has discrete spectrum and $\Omega (\E,\F)= 1$
  \item {\it polynomial growth} if $(\E,\F)$ has discrete spectrum and, for some $c,d>0$,  $\beta_n\le c\cdot n^d$ for all $n\in\mathbb{N}$
  \item{\it intermediate growth} if it has subexponential growth but not polynomial growth.
\end{itemize}
\end{defn}
\begin{lem}
Setting $\gamma_0 =\beta_0$ and
\[
\gamma_n :=\beta_n -\beta_{n-1}=\sharp\{k\in\mathbb{N}: \lambda_k\in (n-1,n]\}\, ,\qquad  n\in\mathbb{N}^*\, ,
\]
and
\[
\Omega^\prime (\E,\F):=\limsup_{n\in\mathbb{N}^*} \sqrt[n]{\gamma_n}
\]
we have
\[
\Omega (\E,\F)=\Omega^\prime (\E,\F)\ge 1\, .
\]
\end{lem}
\begin{proof}
On one hand, by definition, we have $\Omega(\E,\F)\ge \Omega^\prime(\E,\F)$. On the other hand, since, by assumption, $N$ is infinite dimensional and $\sigma(L)$ is discrete, we have $\Omega(\E,\F)\ge \Omega^\prime(\E,\F)\ge 1$. Consider now the following identity involving analytic functions in a neighborhood of $0\in\mathbb{C}$
\[
\sum_{n=0}^\infty\beta_n z^n =(1-z)^{-1}\sum_{n=0}^\infty \gamma_n z^n
\]
and notice that the radius of convergence of the series on the left-hand side is $R=1/\Omega (\E,\F)$, while the radius of convergence of the series on the right-hand side is $R^\prime=1/\Omega^\prime (\E,\F)$  so that $R\le R^\prime\le 1$.
Since $(1-z)^{-1}$ is analytic in the open unit disk centered in $z=0$, the above identity implies that $R\ge R^\prime$ so that $\Omega (\E,\F)\le\Omega^\prime (\E,\F)$.
\end{proof}

\begin{ex} (Spectral growth rate on countable discrete groups).
\item i) On a countable discrete group $\Gamma$, if there exists a proper, c.n.d. function $\ell$, then the associated Dirichlet form $(\E_\ell ,\F_\ell)$ has discrete spectrum $\sigma(L)=\{\ell(g)\in [0,+\infty):g\in\Gamma\}$.\par\noindent
\item ii) On a finitely generated, countable discrete group $\Gamma$, if the length $\ell_S$ corresponding to a finite system of generators $S\subseteq \Gamma$ is negative definite, then the spectral growth rate  $\Omega(\E_{\ell_S},\F_{\ell_S})$ of the corresponding Dirichlet form coincides with growth rate of $(\Gamma ,S)$ (see [deH Ch. VI]).\par\noindent\item iii) Moreover, if $(\Gamma,S)$ has polynomial growth, it has been shown in [CS5] that there exists on $\Gamma$ a proper, c.n.d. function function $\ell$ with polynomial growth. The associated Dirichlet form $(\E_\ell,\F_\ell)$ will have polynomial spectral growth rate.
\end{ex}
\begin{rem}
By a well known bound (see [R Theorem 3.37])
\[
1\le\liminf_n\frac{\beta_{n+1}}{\beta_n}\le\limsup_{n\in\mathbb{N}} \sqrt[n]{\beta_n}  \, ,
\]
if the spectral growth rate is subexponential, then $\liminf_n\frac{\beta_{n+1}}{\beta_n}=1$ so that there exists a subsequence of $\{\frac{\beta_{n+1}}{\beta_n}\}_{n\in\mathbb{N}}$ converging to 1. In other words, the sequence of spectral subspaces $\{E_n\}_{n\in\mathbb{N}}$ corresponding to the interval $[0,n]\subset [0,+\infty)$ admits a subsequence such that
\[
\lim_k\frac{{\rm dim\, }E_{n_k+1}\,}{{\rm dim\,}E_{n_k}}=1\, .
\]
\end{rem}
\noindent
Subexponential growth can be equivalently stated in terms of the nuclearity of the completely Markovian semigroup $\{e^{-tL}: t>0\}$ on $L^2 (N,\o)$:
\begin{lem}
The Dirichlet form $(\E,\F)$ has discrete spectrum and subexponential spectral growth if and only if the Markovian semigroup $\{e^{-tL}: t>0\}$ on $L^2 (N,\o)$ is {\it nuclear, or trace-class}, in the sense that:
\[
{\rm Trace\, }(e^{-tL})=\sum_{k\in\mathbb{N}}e^{-t\lambda_k}<+\infty\qquad t>0\, .
\]
\end{lem}
\begin{proof}
Since
\[
\gamma_0 +\sum_{n\in\mathbb{N}^*}\gamma_n e^{-tn}\le\sum_{k\in\mathbb{N}}e^{-t\lambda_k}\le\gamma_0 +e^t \sum_{n\in\mathbb{N}^*}\gamma_n e^{-tn}\qquad t>0\, ,
\]
the series $\sum_{k\in\mathbb{N}}e^{-t\lambda_k}$ and $\sum_{n\in\mathbb{N}^*}\gamma_n e^{-tn}$ converge or diverge simultaneously. They obviously converge
 for all $t>0$ if and only if $\Omega' (\E,\F)\le 1$.
\end{proof}

\begin{ex}
If on a countable discrete group $\Gamma$, there exists a c.n.d. function $\ell$, such that $\sum_{g\in\Gamma}e^{-t\ell (g)}<+\infty$ for all $t>0$, then $\ell$ is proper, the spectrum of the associated Dirichlet form $(\E_\ell ,\F_\ell)$ coincides with $\{\ell(g)\in [0,+\infty):g\in\Gamma\}$ and it is thus discrete with subexponential growth.
\end{ex}

The following is the main result of this section.

\begin{thm}
Let $(N,\o)$ be a $\sigma$-finite von Neumann algebra endowed with a normal, faithful state on it. If there exists a Dirichlet form $(\E,\F)$ on $L^2(N,\o)$ having subexponential spectral growth, then $N$ is amenable.
\end{thm}
\begin{proof}
Recall that $N$ is amenable if and only if the identity or standard bimodule ${}_NL^2 (N)_N$ is weakly contained in the coarse or Hilbert-Schmidt bimodule $\H_{\rm co}$ (see [Po1]). Consider the completely positive semigroup $\{T_t:=e^{-tL}:t>0\}$ and assume, for simplicity, that the cyclic vector is invariant: $T_t\xi_\o =\xi_\o$ for all $t>0$. Recall (cf. Lemma \ref{cp}) that the complete positivity of $T_t$ provides a binormal state on $N\otimes_{\rm max} N^\circ$ characterized by
\[
\Phi_t :N\otimes_{\rm max} N^\circ\rightarrow\mathbb{C}\qquad \Phi_t(x\otimes y^\circ):=(i_\o (y^*)|T_t i_\o (x))\, .
\]
To compute this state, we consider the spectral representation $T_t =\sum_{k\ge 0}e^{-t\lambda_k} P_k$ (converging strongly) in terms of the rank-one projections $P_k$ on $L^2(N,\o)$ associated to each eigenvalue $\lambda_k$ (repeated according to their multiplicity).
Notice that by Markovianity, the semigroup commutes with the modular conjugation $J_\o$ so that each eigenvector $\xi_k$ may be assumed to be real: $\xi_k =J_\o\xi_k$. We then have
\[
\begin{split}
\Phi_t(x\otimes y^\circ)&=(i_\o (y^*)|T_t i_\o (x)) \\
&=\sum_{k=0}^\infty e^{-t\lambda_k} (i_\o (y^*)|P_k(i_\o (x))) \\
&=\sum_{k=0}^\infty e^{-t\lambda_k} (i_\o (y^*)|(\xi_k|i_\o (x))\xi_k)\\
&=\sum_{k=0}^\infty e^{-t\lambda_k} (\xi_k|i_\o (x))(i_\o (y^*)|\xi_k)\, .
\end{split}
\]
As the series $Z_t :=\sum_{k=0}^\infty e^{-t\lambda_k} \xi_k\otimes \xi_k$ is norm convergent for all $t>0$ by the nuclearity of the semigroup, since $J_\o{}$ is an antiunitary operator on $L^2(N)$, using properties in item  3.4 above we have
\[
\begin{split}
\Phi_t(x\otimes y^\circ)&=\sum_{k=0}^\infty e^{-t\lambda_k} (\xi_k|i_\o (x))(J_\o{}\xi_k|J_\o{} i_\o (y^*))\\
&=\sum_{k=0}^\infty e^{-t\lambda_k} (\xi_k|i_\o (x))(\xi_k|i_{\o^\circ} (y^\circ))\\
&=\sum_{k=0}^\infty e^{-t\lambda_k} (\xi_k\otimes \xi_k|i_\o (x)\otimes i_{\o^\circ} (y^\circ))_{L^2(N,\o)\otimes L^2(N,\o)}\\
&=\Bigl(\sum_{k=0}^\infty e^{-t\lambda_k} \xi_k\otimes \xi_k\Bigl|i_\o (x)\otimes i_{\o^\circ} (y^\circ)\Bigr)_{L^2(N,\o)\otimes L^2(N,\o)}\\
&=\Bigl(Z_t\Bigl|i_{\o\otimes \o^\circ} (x\otimes y^\circ)\Bigr)_{L^2(N\overline{\otimes} N^\circ,\o\otimes\o^\circ)}\, .
\end{split}
\]
Since the symmetric embeddings of von Neumann algebras are continuous when $N\overline{\otimes} N^\circ$ is endowed with the weak$^*$-topology and $L^2(N\overline{\otimes} N^\circ,\o\otimes\o^\circ)$ is endowed with the weak topology, by continuity we have
\[
\Phi_t(z)=\Bigl(Z_t\Bigl|i_{\o\otimes \o^\circ} (z)\Bigr)_{L^2(N\overline{\otimes} N^\circ,\o\otimes\o^\circ)}\qquad z\in N\overline{\otimes} N^\circ\, .
\]
In other words, the linear functional $\Phi_t$ extends as a $\sigma$-weakly continuous linear functional on the spatial tensor product $N\overline \otimes N^\circ$.
 $\Phi_t$ being positive by Lemma 3.7,  there exist a unique positive element $\Omega_t\in L^2_+(N\overline{\otimes} N^\circ,\o\otimes\o^\circ)$ (see [Haa1]) such that
\[
\Phi_t(z)=(i_\o (y^*)|T_t i_\o (x))=\Bigl(\Omega_t |\pi_{\rm co}(z)\Omega_t\Bigr)_{L^2(N\overline{\otimes} N^\circ,\o\otimes\o^\circ)}\qquad z\in N\overline{\otimes} N^\circ
\]
 and the GNS representation of $N\otimes_{\rm max}N^\circ$ associated to $\Phi_t$ coincides with a sub-representation of $\pi_{\rm co}$. In other words, the $N-N$-correspondence $\H_t$ associated to the completely positive map $T_t$ is contained in the coarse $N-N$-correspondence $\H_{\rm co}$ for all $t>0$. Since the semigroup $\{T_t:t>0\}$ is strongly continuous on $L^2(N,\o)$, for all $x\otimes y^\circ\in N\otimes_{\rm max}N^\circ$ we have
\[
\begin{split}
\lim_{t\downarrow 0} \Bigl(\Omega_t |\pi_{\rm co}(x\otimes y^\circ)\Omega_t\Bigr)_{L^2(N,\o)\otimes L^2(N,\o)}&= (i_\o (y^*)|i_\o (x))_{L^2(N,\o)} \\
&=(\Delta_\o^{\frac{1}{4}}y^*\xi_\o |\Delta_\o^{\frac{1}{4}}x\xi_\o ) \\
&=(\Delta_\o^{\frac{1}{2}}y^*\xi_\o |x\xi_\o ) \\
&=(J_\o{} y\xi_\o|x\xi_\o) \\
&=(J_\o{} yJ_\o{}\xi_\o|x\xi_\o) \\
&=(\xi_\o|J_\o{} y^*J_\o{}x\xi_\o) \\
&=(\xi_\o|xJ_\o{} y^*J_\o{}\xi_\o) \\
&=(\xi_\o |x\xi_\o y) \\
&=(\xi_\o |\pi_{\rm id}(x\otimes y^\circ)\xi_\o)
\end{split}
\]
and by continuity
\[
\lim_{t\downarrow 0} \Bigl(\Omega_t |\pi_{\rm co}(z)\Omega_t\Bigr)_{L^2(N,\o)\otimes L^2(N,\o)}=(\xi_\o |\pi_{\rm id}(z)\xi_\o)\qquad z\in N\otimes_{\rm max}N^\circ\, .
\]
This proves that the identity correspondence $\H_{\rm id}$ is weakly contained in the coarse correspondence $\H_{\rm co}$ and thus $N$ is amenable at least if the semigroup leaves the cyclic vector invariant.
To deal with the general case, remark first that, by strong continuity, we have that \\$\lim_{t\downarrow 0}(\xi_\o|T_t\xi_\o)=\|\xi_\o\|^2=1$ and there exist $t_0 >0$ such that $(\xi_\o|T_t\xi_\o)>0$ for all $0<t<t_0$. Applying the argument above to the binormal states
\[
\widetilde\Phi_t (x\otimes y^\circ):=\frac{1}{(\xi_\o|T_t\xi_\o)}(i_\o (y^*)|T_t i_\o (x))\qquad x\otimes y^\circ\in N\otimes_{\rm max}N^\circ\, ,\qquad 0<t<t_0
\]
we get the amenability of $N$ even in the general situation.
\end{proof}

\begin{rem}
i) The above result implies that if the von Neumann algebra $N$ is not amenable, then any Dirichlet form $(\E,\F)$ with respect to any normal, faithful state $\o$ has exponential growth rate $\Omega(\E,\F)>1$, i.e. its sequence of eigenvalues has exponentially growing distribution.
ii) Conversely, it is an open question whether there exist amenable von Neumann algebras on which every Dirichlet form has exponential growth. The analogy with discrete groups suggests that the the answer is likely positive.

\end{rem}

The following one is a generalization of a result of Guentner-Kaminker [GK].
\begin{cor}
Let $\Gamma$ be a countable discrete group, $\lambda:\Gamma\rightarrow \mathcal{B}(l^2(\Gamma))$ be its left regular representation, $L(\Gamma)$ its associated von Neumann algebra and $\tau$ its trace state. If there exists a Dirichlet form $(\E,\F)$ on $L^2(L(\Gamma),\tau)$ having subexponential spectral growth, then the group $\Gamma$ is amenable.
\end{cor}
\begin{proof}
Under the assumptions, the group von Neumann algebra $L(\Gamma)$ is amenable by the above theorem. Hence by a well known result of A. Connes, the group $\Gamma$ is amenable.
\end{proof}

\begin{ex}(Free orthogonal quantum groups)
On the von Neumann algebra $L^\infty (O^+_2,\tau)$ of the {\it free orthogonal quantum group} $O^+_2$ with respect to its Haar state $\tau$, it has been constructed in [CFK] a Dirichlet form with an explicitly computed discrete spectrum of polynomial growth (and spectral dimension $d:=\limsup_n\frac{\ln \beta_n}{\ln n}=3$). Applying the theorem above one obtains a proof of the amenability of $L^\infty (O^+_2,\tau)$, a result which has been proved by M. Brannan [Bra].
\end{ex}

\section{Relative amenabiltity of inclusions of finite von Neumann algebras}
In this section we extend the previous result to the {\it relative amenability} of inclusions of finite von Neumann algebras $B\subseteq N$, as defined by S. Popa [Po 1,2]. This extension is based on the properties of the relative tensor product of Hilbert bimodules and on the properties of the {\it basic construction}, which we will presently recall
(see [Chr], [GHJ], [J], [SiSm]).
\subsection{Basic construction of finite inclusions}
Let $N$ be a von Neumann algebra admitting a normal faithful trace state $\tau$ and $1_N\in B\subseteq N$ a von Neumann subalgebra with the same identity (see [Chr], [J1], [Po1], [PiPo], [SiSm]).
\vskip0.2truecm\noindent
Recall that the relative tensor product $L^2(N,\tau)\otimes_B L^2(N,\tau)$ over $B$ of the $N$-$B$-bimodule ${}_N L^2(N,\tau)_B$ by the $B$-$N$-bimodule ${}_B L^2(N,\tau)_N$, constructed in [S1], is isomorphic, as an $N-N$-bimodule, to the $N$-$N$-correspondence $\H_B$ associated to the conditional expectation $E_B:N\to N$ from $N$ onto $B$. The latter being generated by the GNS construction applied to the binormal state
\[
\Phi_B :N\otimes_{\rm max} N^\circ\rightarrow\mathbb{C}\qquad \Phi_B (x\otimes y^\circ):=\tau(E_B(x)y)\, .
\]

According to S. Popa (see [Po1,2]), the inclusion $B\subseteq N$ is said to be {\it relatively amenable} if the standard bimodule ${}_N L^2(N,\tau)_N$ is weakly included in the relative coarse bimodule ${}_N L^2(N,\tau)\otimes_B L^2(N;\tau)_N$.

Let $e_B$ be the oerthogonal projection in $\mathcal{B}(L^2(N,\tau))$ from $L^2(N,\tau)$ onto $L^2(B,\tau)$ and consider the {\it basic construction} $\langle N,B\rangle$, i.e. the von Neumann algebra in $\mathcal{B}(L^2(N,\tau))$ generated by $N$ and the projection $e_B$.
For example, if $B=\mathbb{C}1_N$ then $\langle N,B\rangle=\mathcal{B}(L^2(N,\tau))$ and when $B=N$ then $\langle N,B\rangle=N$.
\par\noindent
Denoting by $\xi_\tau\in L^2(N,\tau)$ the cyclic vector representing $\tau$ one has \[
e_B(x\xi_\tau)=E_B(x)\xi_\tau\, ,\quad e_B xe_B =E_B (x)e_B\qquad x\in N\, .
\]
It can be shown that an element $x\in N$ commutes with the projection  $e_B$ if and only if $x\in B$.
Moreover, ${\rm span}(Ne_B N)$ is weakly$^*$-dense in $\langle N,B\rangle$ and $e_B \langle N,B\rangle e_B=Be_B$. It can be shown that
\[
\langle N,B\rangle=(JBJ)^\prime\subseteq \mathcal{B}(L^2(N,\tau))
\]
so that $\langle N,B\rangle$ is semifinite since $B$ is finite. In particular, there exists a unique normal, semifinite faithful trace $\rm Tr$ characterized by
\[
{\rm Tr}(xe_B y)=\tau(xy)\qquad x,y\in N\, .
\]
and there exists also a unique $N-N$-bimodule map $\Phi$ from ${\rm span}(Ne_B N)$ into $N$ satisfying
\[
\Phi (xe_B y)=xy\quad x,y\in N\, ,\qquad {\rm Tr}=\tau\circ\Phi\, .
\]
The map $\Phi$ extends to a contraction between the $N$-$N$-bimodules $L^1(\langle N,B\rangle,{\rm Tr})$ and $L^1(N,\tau)$ and satisfies
\[
e_B X=e_B \Phi (e_B X)\qquad X\in \langle N,B\rangle\, .
\]
Moreover, $\Phi(e_B X)\in L^2(\langle N,B\rangle,{\rm Tr})$ for all $X\in \langle N,B\rangle$. These properties enable us to prove that the identity correspondence $L^2(\langle N,B\rangle, {\rm Tr})$ of the algebra $\langle N,B\rangle$ reduces to the relative correspondence $\H_B$ when restricted to the subalgebra $N\subseteq \langle N,B\rangle$.
\vskip0.2truecm\noindent
The following proposition is well known, we give the proof for sake of completeness.
\begin{prop}
The $N$-$N$-correspondences $\H_B$ and $L^2(\langle N,B\rangle, {\rm Tr})$ are isomorphic. In particular, the binormal state is given by
\[
\Phi_B(x\otimes y^\circ)=(e_B|xe_B y)_{L^2(\langle N,B\rangle, {\rm Tr})}\qquad x,y\in N
\]
so that the cyclic vector representing the state $\Phi_B$ is $e_B\in L^2(\langle N,B\rangle, {\rm Tr})$.
\end{prop}
\begin{proof}
Let us consider the map $\Psi$ defined on the domain
\[
D(\Psi):={\rm span}\{[x\otimes y^\circ]_{\H_B}: x,y\in N\}
\]
by
\[
\Psi:D(\Psi)\rightarrow L^2(\langle N,B\rangle, {\rm Tr})\qquad \Psi([x\otimes y^\circ]_{\H_B}):= xe_B y\qquad x,y\in N\,.
\]
 Here $[x\otimes y^\circ]_{\H_B}$ denotes the element of $\H_B$ image of the elementary tensor product $x\otimes y^\circ$, in the GNS construction of the state $\Phi_B$. The map is well defined because $\|e_B\|^2 ={\rm Tr}(e_B)=\tau (\Phi (e_B))=\tau (1_N)=1$ and $\|xe_B y\|_2\le \|x\|\cdot\|y\|\cdot \|e_B\|_2 =\|x\|\cdot\|y\|$. By the definition of the Hilbert space $\H_B$, the map $\Psi$ is densely defined. For $x,y\in N$ we have
\[
\begin{split}
\|\Psi([x\otimes y^\circ]_{\H_B})\|^2_2 &=\|xe_B y\|^2_2 \\
&={\rm Tr}((xe_B y)^*(xe_B y)) \\
&={\rm Tr}(y^*e_B x^*xe_B y) \\
&={\rm Tr}(y^*e_B e_B x^*xe_B e_B y) \\
&={\rm Tr}((e_B x^*x e_B) (e_B yy^* e_B)) \\
&={\rm Tr}(E_B (x^*x)e_B E_B (yy^*)e_B) \\
&={\rm Tr}(E_B (x^*x) e_B E_B (yy^*)) \\
&=\tau(E_B (x^*x) E_B (yy^*)) \\
&=\tau(E_B (E_B (x^*x) yy^*)) \\
&=\tau(E_B (x^*x) yy^*) \\
&=\Phi_B (x^* x\otimes (yy^*)^\circ) \\
&=\Phi_B (x^* x\otimes (y^*)^\circ y^\circ) \\
&=\Phi_B ((x^*\otimes (y^*)^\circ)(x\otimes y^\circ)) \\
&=\Phi_B ((x\otimes y^\circ)^*(x\otimes y^\circ)) \\
&=\|[x\otimes y^\circ]_{\H_B}\|^2_{\H_B}\, . \\
\end{split}
\]
By polarization, for all $\{x_j\}_{j=1}^n\subset N$ we have also
\[
(\Psi([x_j\otimes y_j^\circ]_{\H_B}|\Psi([x_k\otimes y_k^\circ]_{\H_B})_2=([x_j\otimes y_j^\circ]_{\H_B}|[x_k\otimes y_k^\circ]_{\H_B})_{\H_B}\qquad j,k=1\, ,\cdots\ ,n\, .
\]
Consider now $\nu =\sum_{k=1}^n [x_k\otimes y_k^\circ]_{\H_B}\in D(\Psi)$ so that $\nu^*\nu =\sum_{j,k=1}^n [x_j^* x_k\otimes (y_ky_j^*)^\circ]_{\H_B}\in D(\Psi)$ and then
\[
\|\Psi(\nu)\|^2_2 = \sum_{j,k=1}^n ([x_j\otimes y_j^\circ]_{\H_B}|[x_k\otimes y_k^\circ]_{\H_B})_{\H_B}=(\nu|\nu)_{\H_B}=\|\nu\|^2_{\H_B}\, .
\]
Hence the map $\Psi$ extends to an isometry from $\H_B$ into $L^2(\langle N,B\rangle, {\rm Tr})$ which is clearly an $N-N$-bimodule map. Since ${\rm Im}(\Psi)={\rm span}(Ne_B N)$ is weakly$^*$-dense in $\langle N,B\rangle$, it is also dense in $L^2(\langle N,B\rangle, {\rm Tr})$. By the isometric property we have that ${\rm Im}(\Psi)$ is closed so that $\Psi$ is a surjective isometry. Finally, for $x,y\in N$ we compute
\[
\begin{split}
(e_B|xe_B y)_{L^2({\rm Tr})}&={\rm Tr}(e_Bxe_B y) \\
&={\rm Tr}(E_B(x)e_By) \\
&=\tau(\Phi (E_B(x)e_By)) \\
&=\tau(E_B(x)y) \\
&=\Phi_B(x\otimes y^\circ)\, . \\
\end{split}
\]
\end{proof}

\begin{defn}({\bf $B$-invariant Dirichlet forms}).
Let $N$ be a von Neumann algebra admitting a normal faithful tracial state $\tau$ and $1_N\in B\subseteq N$ a von Neumann subalgebra.
\par\noindent
A Dirichlet form $(\E,\F)$ on $L^2(N,\tau)$ is said to be a $B$-{\it invariant} if
\[
b\F \subseteq\F\, ,\quad \E(b\xi|\xi)=\E(\xi|b^*\xi)\qquad b\in B, \quad \xi\in\F
\]
and
\[
\F b \subseteq\F\, ,\quad \E(\xi b|\xi)=\E(\xi b^*|\xi)\qquad b\in B, \quad \xi\in\F\, .
\]
Since, by definition, a Dirichlet form is $J$-real, the above two properties are in fact equivalent.
\vskip0.2truecm\noindent
In terms of the associated nonnegative, self-adjoint operator $(L,D(L))$, $B$-invariance means that the resolvent family $\{(\lambda +L)^{-1}: \lambda >0\}$ is $B$-bimodular for some and hence all $\lambda >0$
\[
\begin{split}
(\lambda +L)^{-1}(b\xi)&=b((\lambda +L)^{-1}\xi) \\
(\lambda +L)^{-1}(\xi b)&=((\lambda +L)^{-1}\xi)b\qquad \xi\in L^2(N,\tau)\, ,b\in B\, , \\
\end{split}
\]
or that, alternatively, the semigroup $\{e^{-tL}:t>0\}$ is for some and hence all $t>0$ a $B$-bimodular map
\[
\begin{split}
e^{-tL}(b\xi)&=b(e^{-tL}\xi) \\
e^{-tL}(\xi b)&=(e^{-tL}\xi)b\qquad \xi\in L^2(N,\tau)\, ,b\in B\, . \\
\end{split}
\]
Since the Markovianity of the Dirichlet form implies that the semigroup and the resolvent commute with the modular conjugation $J$, we have that the $B$-invariance of the Dirichlet form provides
that the semigroup and the resolvent belong to the relative commutant of $B$ in the basic construction $\langle N,B\rangle$:
\[
(\lambda+L)^{-1},\,\, e^{-tL}\in (JBJ)^\prime\cap B^\prime=\langle N,B\rangle\cap B^\prime\qquad t>0\, ,\quad \lambda >0\, .
\]
\end{defn}
\begin{defn}({\bf Relative discrete spectrum})
We say that $(\E,\F)$ or $(L,D(L))$ have {\it discrete spectrum relative to the inclusion} $B\subseteq N$ if the Markovian semigroup, or equivalently the resolvent, belongs to the {\it compact ideal space} $\mathcal{J}(\langle N,B\rangle)$ ([J], [Po2], [SiSm]) of the basic construction, generated by projections in $\langle N,B\rangle$ having finite trace:
\[
e^{-tL}\in \mathcal{J}(\langle N,B\rangle)\qquad {\rm for\,\,some\,\, and\,\, hence\,\, all}\,\, t>0\, ,
\]
\[
(\lambda+L)^{-1}\in \mathcal{J}(\langle N,B\rangle)\qquad {\rm for\,\,some\,\, and\,\, hence\,\, all}\,\, \lambda >0\, ,
\]
Another way to state it is that the spectrum of $(L,D(L))$ is a discrete subset of $\mathbb{R}_+$ and that each eigenprojection has finite $\rm Tr$ trace.
\end{defn}
\begin{defn}({\bf Relative spectral growth rate of Dirichlet forms})
Let $N$ be a von Neumann algebra admitting a normal faithful tracial state $\tau$ and $1_N\in B\subseteq N$ a von Neumann subalgebra. A Dirichlet form $(\E,\F)$ on $L^2(N,\tau)$ which is $B$-{\it invariant} is said to be have
\begin{itemize}
  \item {\it exponential spectral growth relative to $B\subseteq N$} if ${\rm Tr}(e^{-tL})=+\infty$ for some $t>0$;
  \item {\it subexponential spectral growth relative to $B\subseteq N$} if ${\rm Tr}(e^{-tL})<+\infty$ for all $t>0$.
\end{itemize}
Notice that, if ${\rm Tr}(e^{-tL})<+\infty$ for some $t>0$, then $e^{-tL}\in \mathcal{J}(\langle N,B\rangle)$ so that $(L,D(L))$ has discrete spectrum relative to the inclusion $B\subseteq N$. This applies, in particular, to $B$-{\it invariant} Dirichlet forms $(\E,\F)$  with subexponential spectral growth relative to $B\subseteq N$ which thus have necessarily discrete spectrum relative to the inclusion $B\subseteq N$.
\end{defn}

\begin{rem}
Let $\mathbb{E}^L$ be the spectral measure of the self-adjoint operator $\big (L,D(L)\big )$. If the Dirichlet form is $B$-invariant then $\mathbb{E}^L$ takes its values in the class of projections of the von Neumann algebra $\langle N,B\rangle$ and we can consider the positive measure $\nu_B^L :={\rm Tr}\circ \mathbb{E}^L$ on $[0,+\infty)$, supported by the spectrum ${\sigma(L)}$. In the framework of quantum statistical mechanics, where the operator $L$ may represent the total energy observable, the measure $\nu_B^L$ acquires the meaning of "density of states" in the sense that $\nu_B^L(\O)$ measures the number (relatively to $B$) of allowed energy levels located in a measurable subset $\O\subset \sigma(L)$. In this case the subspace $L^2(B,\tau)\subset L^2(N,\tau)$ may represent the manifold of ground states corresponding to the minimal allowable energy level (see also Example 4.7 below). The measure $\mu_B^L (d\lambda):=\lambda\nu_B (d\lambda)$ has then the meaning of "spectral energy density" in the sense that $\mu_B^L(\O)$ measures the energy of the system in a situation where all the allowed energy levels in $\O\subset \sigma(L)$ are occupied.
\par\noindent
The subexponential spectral growth condition (relatively to $B\subseteq N$) can be rephrased in terms of the Laplace Transform $\hat{\nu}_B^L$ saying that its abscissa of convergence vanishes. In this situation $\hat{\nu}_B^L(\beta)={\rm Tr\,}(e^{-\beta L})$ is called the {\it partition function} of the system, it is defined for all $\beta>0$ and the variable $\beta$ is interpreted as the {\it inverse temperature}.
The subexponential spectral growth condition also allows to consider the so called {\it Gibbs} normal states, defined by $\Phi_\beta (A):=\frac{{\rm Tr}(Ae^{-\beta L})}{{\rm Tr}(e^{-\beta L})}$, on the von Neumann (observable) algebra $\langle N,B\rangle$, for any fixed value of the inverse temperature $\beta >0$. These states possess properties by which they can be regarded as equilibria of the system at the fixed value of the inverse temperature. Finally,
notice that the subexponential spectral growth condition is equivalent to the requirement that the mean energy $\Phi_\beta (L):=\frac{{\rm Tr}(Le^{-\beta L})}{{\rm Tr}(e^{-\beta L})}$ of the system is finite for any $\beta>0$ (see [BR]).
\end{rem}
The following is the main result of this section.
\begin{thm}
Let $N$ be a von Neumann algebra admitting a normal faithful tracial state $\tau$ and $1_N\in B\subseteq N$ a von Neumann subalgebra.
\par\noindent
If there exists a $B$-invariant Dirichlet form $(\E,\F)$ on $L^2(N,\tau)$ having subexponential spectral growth relatively to $B\subseteq N$, then the inclusion $B\subseteq N$ is amenable.
\end{thm}
\begin{proof}
Let us check first the following identity
\[
(T^*|xe_B y)_{L^2({\rm Tr})}=(i_\tau (y^*) | T (i_\tau (x))_{L^2(\tau))}\qquad T\in \langle N,B\rangle\cap L^2(\langle N,B\rangle,{\rm Tr})\, , \quad x,y\in N\, .
\]
As ${\rm span}(Ne_B N)$ is weakly$^*$ dense in $\langle N,B\rangle$, it is enough to prove the identity for $T\in Ne_B N$.
If $T =u e_B v$ for some $u,v\in N$ we have
\[
e_ByT xe_B = e_B yu e_B v xe_B =(e_B yu e_B)( e_B v xe_B)= E_B (yu) e_B E_B (v x)e_B
\]
and then
\[
\begin{split}
(T^*|xe_B y)_{L^2({\rm Tr})} &={\rm Tr}(T xe_By) \\
&={\rm Tr}(e_ByT xe_B) \\
&=\tau(\Phi (e_ByT xe_B)) \\
&=\tau(\Phi (E_B (yu) e_B E_B (v x)e_B)) \\
&=\tau(E_B (yu) \Phi (e_B E_B (v x)e_B)) \\
&=\tau(E_B (yu) \Phi (E_B (v x)e_B)) \\
&=\tau(E_B (yu) E_B (v x)) \\
&=\tau(E_B (yu E_B (v x)) ) \\
&=\tau(yu E_B (v x)) \\
&=(u^* y^*\xi_\tau | E_B (v x)\xi_\tau)_{L^2(\tau)} \\
&=(i_\tau (y^*) | u E_B (v x)\xi_\tau)_{L^2(\tau)} \\
&=(i_\tau (y^*) | u e_B (v (x\xi_\tau)))_{L^2(\tau)} \\
&=(i_\tau (y^*) | u e_B v (i_\tau (x))_{L^2(\tau)} \\
&=(i_\tau (y^*) | T (i_\tau (x))_{L^2(\tau)} \\
\end{split}
\]
so that the identity holds true. Under the hypothesis of subexponential spectral growth, we have that $T_t :=e^{-tL}\in L^2(\langle N,b\rangle,{\rm Tr})\cap L^1(\langle N,b\rangle,{\rm Tr})$ for all $t>0$. Applying the above identity, we have that the binormal states
\[
\Phi_t: N\otimes_{\rm max}N^\circ \rightarrow\mathbb{C}\qquad \Phi_t (x\otimes y^\circ):=\frac{1}{(\xi_\tau|T_t\xi_\tau)_{L^2(N,\tau)}}(i_\tau (y^*) | T_t (i_\tau (x))_{L^2(N,\tau)}\, ,
\]
well defined, by strong continuity of the semigroup, for $t$ sufficiently close to zero, may be represented for $t>0$ as $\Phi_t (x\otimes y^\circ)=\frac{1}{(\xi_\tau|T_t\xi_\tau)_{L^2(N,\tau)}}(T_t|xe_B y)_{L^2(\langle N,b\rangle,{\rm Tr})}$.


By the identity above, $\Phi_t$ extends as a normal state on the von Neumann algebra generated by the left and right representations of $N$ in $L^2(\langle N,B\rangle,{\rm Tr})$. The
 $N$-$N$-correspondence $\H_t$ generated by $\Phi_t$ is thus a  sub-correspondence of a multiple of the $N$-$N$-correspondence $L^2(\langle N,B\rangle,{\rm Tr})$.
 Since the semigroup $\{T_t:t>0\}$ strongly converges to the identity operator on $L^2(N,\tau)$, we obtain that the trivial correspondence from $N$ to $N$ is weakly contained in the relative correspondence $\H_B$.
\end{proof}
\begin{ex}(Minimal and maximal inclusions of a Dirichlet form)
Let $N$ be a von Neumann algebra and $\tau$ a normal, faithful, tracial state and let $(\E,\F)$ be a Dirichlet form on $L^2(N,\tau)$ with associated self-adjoint operator $(L,D(L))$.
\par\noindent
Assume that $\inf\sigma(L)=0$ and that this is an eigenvalue (not necessarily of finite multiplicity). The spectral projection $P_0$ onto the eigenspace corresponding to the Borel subset $\{0\}\subset [0,+\infty)$ can be represented as the strong limit $P_0=\lim_{t\to +\infty}e^{-tL}$ . Hence $P_0$ is a completely Markovian projection, so that there exists a von Neumann subalgebra $B_{\rm min}\subseteq N$ such that $P_0=e_{B_{\rm min}}$. Obviously the associated Markovian semigroup is $B_{\rm min}$-bimodular and the Dirichlet form is $B_{\rm min}$-invariant.
\vskip0.2truecm\noindent
Alternatively, one can consider the inclusion $B_{\rm max}\subset N$ where $B_{\rm max}:=\{T_t:t>0\}^\prime\cap N$ is the relative commutant of the Markovian semigroup in $N$. Notice that by the Spectral Theorem $B_{\rm max}=\{T_t\}^\prime\cap N$ for all $t>0$. Obviously the associated Markovian semigroup is $B_{\rm max}$-bimodular and the Dirichlet form is $B_{\rm max}$-invariant.


\begin{prop}
Let $(N,\tau)$ be a finite von Neumann algebra with faithful, normal trace. Let $(\E,F)$ a Dirichlet form on $L^2(N,\tau)$ with generator $(L,D(L))$
having pure point spectrum made by distinct, isolated eigenvalues $\sigma(L):=\{\lambda_0 <\lambda_1< \lambda_2 <\cdots\}$ and assume $\lambda_0 :=\inf\sigma(L)=0$.
\par\noindent
Then $(\E,F)$ has discrete spectrum relative to $B_{\rm min}$ (resp. $B_{\rm max}$) if and only if each eigenspace $E_\lambda\subset L^2(N,\tau)$, $\lambda\in \sigma(L)$, has finite coupling constant
${\rm dim\,}_{B_{\rm min}}(E_\lambda)<+\infty$ (resp. ${\rm dim\,}_{B_{\rm max}}(E_\lambda)<+\infty$) relative to $B_{\rm min}$ (resp. $B_{\rm max})$.
\end{prop}
\end{ex}
\noindent
Remark that the finite coupling constant ${\rm dim\,}_{B_{\rm min}}(E_\lambda)<+\infty$ (resp. ${\rm dim\,}_{B_{\rm max}}(E_\lambda)<+\infty$) relative to $B_{\rm min}$ (resp. $B_{\rm max})$ is well defined for any eigenvalue $\lambda\in \sigma(L)$ because any eigenspace $E_\lambda$ is obviously a left (and also right) $B_{\rm min}$-module (resp. $B_{\rm max}$-module). We refer to [GHJ Section 3.2] for the definition and properties of the Murray-von Neumann coupling constant.

\begin{ex}
Let $K<\Gamma$ be an inclusion of countable, discrete groups and let $L(K)\subset L(\Gamma)$ be the inclusion of the finite von Neumann algebras generated by $K$ and $\Gamma$, respectively. Their standard spaces coincide with $l^2(K)$ and $l^2(\Gamma)$ respectively and the projection $e_{L(K)}$ coincides with the projection from $l^2(\Gamma)$ onto its subspace $l^2 (K)$.
\par\noindent
Let $\ell :\Gamma\to [0,+\infty)$ be a c.n.d. function. The Dirichlet form $(\E_\ell,\F_\ell)$ associated to $\ell$ (introduced in Section 3.3) is $L(K)$-invariant if and only if $\ell$ vanishes on $K$ or, equivalently, if $\ell$ is a right $K$-invariant function. In this situation we have:
\end{ex}

\begin{prop}
Let $\Gamma$ be a countable, discrete group and let $L(\Gamma)$ be its left von Neumann algebras. Let $\ell :\Gamma\to [0,+\infty)$ be a c.n.d. function and $(\E_\ell,\F_\ell)$ the associated Dirichlet form. Denote by $H:=\{s\in\Gamma:\ell(s)=0\}$ the subgroup where $\ell$ vanishes. We then have
\item i) $B_{\rm min}=B_{\rm max}=L(H)$;
\item ii) If $K$ is a subgroup of $G$, then
 $(\E_\ell,\F_\ell)$ is $L(K)$-invariant if and only if $K<H$. In this case\,:
 \par\noindent
 ii.a) $\ell$ is $L(K)$-biinvariant
\par\noindent
ii.b)  $(\E_\ell,\F_\ell)$ has discrete spectrum relative to $L(K)\subset L(\Gamma)$ if and only if the function
\[
{\ell}_{G/K}:G/K\to [0,+\infty)\qquad {\ell}_{G/K}({\widetilde s}):=\ell (s)
\]
defined for ${\widetilde s}=sK\in G/K$, is proper.
\par\noindent
ii.c) If, for any $t>0$, $\sum_{\widetilde s\in G/K} e^{-t\ell_{G/K}(\widetilde s)}<+\infty$, then the inclusion $L(K)\subset L(G)$ is amenable.
\end{prop}

\begin{proof}
i) Let $x=\sum_{t\in \Gamma}x(t)\lambda(t)\in B_{\rm max}=\{T_t:t>0\}^\prime\cap L(\Gamma)$. We then have 
\[
x\delta_e=x(I+L)^{-1}\delta_e = (I+L)^{-1}x\delta_e\ ,
\]
which implies  $0=L(x\delta_e)=L(\sum_{t\in \Gamma}x(t)\lambda(t)\delta_e)=L(\sum_{t\in \Gamma}x(t)\delta_t)=\sum_{t\in \Gamma}x(t)\ell(t)\delta_t$. So that $x(t)\ell(t)=0$ for all $t\in\Gamma$ which in turn implies $x\in L(H)=B_{\rm min}$. The reverse inclusion is obvious.
\par\noindent
ii) follows from the arguments of the example above. For ii) b) just notice that $\lambda(s)e_{L(K)}\lambda(s)^{-1}$ is the orthogonal projection $P_{sK}$ onto the subspace $l^2(sH)$. Hence the eigenspace $E_\lambda$ corresponding to the eigenvalue $\lambda\in \sigma(L)$ is given by $\bigoplus_{\widetilde s\in G/K\, ,\,\, \ell(s)=\lambda} l^2(sK)$. Hence, $L$ will have discrete spectrum relative to $K$ if and only if each of these sums is finite (i.e. for all $\lambda$) and the set of values of $\ell$ is discrete, i.e. $\ell^{-1}(\{\lambda\})/K$ is finite in $G/K$, i.e. ${\ell}_{G/K}:G/K\to [0,+\infty)$ is proper.

\par\noindent For ii.c), note that in the basic construction for $B=L(K)\subset N=L(G)$, $P_{sK}=\lambda(s)e_B\lambda(s)$ belongs to $\langle N,B \rangle$ and has trace $1$. Hence $Tr(e^{-tL})=\sum_{\widetilde s\in G/K} e^{t\ell(s)}$ for all $t>0$.
\end{proof}
\begin{rem}
\item i) If the Dirichlet form $(\E_\ell,\F_\ell)$ has discrete spectrum relative to $L(K)$ then \par\noindent a) the function ${\ell}_{G/H}:G/H\to [0,+\infty)$ is proper and $(\E_\ell,\F_\ell)$ has discrete spectrum relative to $L(H)$\,; \par\noindent b) $\ell_{G/K}$ being left $K$-invariant, hence constant on left $K$-cosets, and proper, left $K$-cosets in $G/K$ must be finite sets. In other words, $K$ is quasi-normal in $G$.
\item ii) On the other hand, if $(\E_\ell,\F_\ell)$ has discrete spectrum relative to $L(H)$ then the function ${\ell}_{G/K}:G/K\to [0,+\infty)$ will be constant onto the right $H$-coset in $G/K$. Thus $(\E_\ell,\F_\ell)$ has discrete spectrum relative to $L(K)$ if and only if each right $H$-coset in $G/K$ is a finite union of $K$-cosets, which happens if and only if $K$ has finite index in $H$, i.e. when the homogeneous space $H/K$ is finite.
\end{rem}

\section{A spectral approach to the Relative Haagerup property}

As already mentioned in the Introduction, in a recent work [CaSk], M. Caspers and A. Skalski characterized von Neumann algebras having Property (H) in terms of the existence of a Dirichlet form with discrete spectrum. In the spirit of the previous section, we extend their result to {\it relative property (H)}, as defined by S. Popa [Po1,2], for inclusions of von Neumann algebras, using a completely different approach. We will make use of the following well known properties:

\begin{richiamo}
Let $(N,\tau)$ be a von Neumann algebra endowed with a normal, faithful trace and let $\varphi:N\to N$ be a completely positive, normal contraction such that $\tau\circ\varphi\le\tau\, .$ Then\par\noindent
i) there exists a contraction $T_\varphi\in \mathcal{B}(L^2(N,\tau))$ characterized by
\[
T_\varphi (x\xi_\tau) = \varphi (x)\xi_\tau\qquad x\in N\, ;
\]
ii) there exists a completely positive, normal contraction $\varphi^*:N\to N$ such that
\[
T_{\varphi^*}=(T_\varphi)^*
\]
or, more explicitly,
\[
(\varphi^* (y)\xi_\tau|x\xi_\tau)=(y\xi_\tau|\varphi (x)\xi_\tau)\qquad x,y\in N\, .
\]
\end{richiamo}

\begin{defn}([Po1,2])
Let $N$ be a finite von Neumann algebra and $B\subseteq N$ a von Neumann subalgebra. Then $N$ is said to have {\it Property (H) relative to $B$} if there exist a normal, faithful tracial state $\tau$ on $N$ and a net $\{\varphi_i:i\in I\}$ of normal completely positive, $B$-bimodular maps on $N$ satisfying the conditions
\begin{enumerate}
  \item [i)] $\tau\circ\varphi_i\le \tau$
  \item [ii)] $T_{\varphi_i}\in \mathcal{J}(\langle N,B\rangle)$
  \item [iii)] $\lim_{i\in I} \|x\xi_\tau-T_{\varphi_i} (x\xi_\tau)\|_2 =0$ for all $x\in N$.
\end{enumerate}
In this definition $\mathcal{J}(\langle N,B\rangle)$ is the compact ideal space, i.e. the norm closed ideal generated by projections with finite trace in $\langle N,B\rangle$ and $T_{\varphi_i}$ is the operator defined in item 5.1 above.
\end{defn}

By a remark of S. Popa [Po2], the maps $\varphi_i$ in the definition above can be chosen to be contractions. In the following we shall always assume this property for approximating nets of the identity map of a von Neumann algebra.

\begin{thm}\label{propH}
Let $N$ be a finite von Neumann algebra with countably decomposable center and faithful tracial state $\tau$. Let $B\subseteq N$ be a sub-von Neumann algebra such that $L^2(N,\tau)$, as $B$-module, admits a countable base. Then the following properties are equivalent
\begin{itemize}
\item [i)] $N$ has Property (H) relative to $B$
\item [ii)] there exists a $B$-invariant Dirichlet form $(\E,\F)$ on $L^2(N,\tau)$ with discrete spectrum relative to $B$.
\end{itemize}
\end{thm}
\begin{proof}
Assume that there exists a Dirichlet form $(\E,\F)$ on $L^2(N,\tau)$ with discrete spectrum relative to $B$. Hence, the associated generator $(L,D(L))$ has its resolvent in the compact ideal space: $(\lambda +L)^{-1}\in \mathcal{J}(\langle N,B\rangle)$. Then for all $\lambda >0$, $S_\lambda :=\lambda(\lambda +L)^{-1}\in \mathcal{J}(\langle N,B\rangle)$. Moreover, any $S_\lambda$ is Markovian on $L^2(N,\tau)$ which implies that there exists a completely positive contraction $\varphi_\lambda:N\to N$ determined by $S_\lambda (x\xi_\tau)=\varphi_\lambda (x)\xi_\tau$ for $x\in N$. Since the $S_\lambda$ are self-adjoint on $L^2(N,\tau)$, the $\varphi_\lambda$ are symmetric with respect to the trace: $\tau (\varphi_\lambda (x)y)=\tau(x\varphi_\lambda (y))$ for all $x,y\in N$. This implies that
\[
\tau (\varphi_\lambda (x))=\tau (\varphi_\lambda (1_N)x)\le \tau(x)\qquad x\in N_+\, .
\]
Last condition iii) in Definition 5.2 above comes from the strong continuity of the resolvent:
\[
\lim_{\lambda\to +\infty} \|\xi-S_\lambda\xi\|_2=0\qquad \xi\in L^2(N,\tau)\, .
\]
The theorem is proved in the "if" direction. In the reverse direction, let us suppose that $B\subseteq N$ is an inclusion with relative property (H) and that $L^2(N,\tau)$ is separable as $B$-module. Let $\{\varphi_n:n\in \mathbb{N}\}$ be a sequence of normal, completely positive, $B$-bimodular contractions of $N$, satisfying the conditions of the definition above. By [Po2 Proposition 2.2] such a sequence always exists. Each $\varphi_n$ extends by $T_n(x\xi_\tau):=\varphi_n (x)\xi_\tau$ to a $B$-bimodular contraction $T_n$ of $L^2(N,\tau)$, which belongs to the compact ideal space $\mathcal{J}(\langle N,B\rangle)$. It is also completely positive with respect to the standard positive cone $L^2_+(N,\tau)$ and its matrix amplifications. It is easy to check that the maps $\varphi_n^*$ appearing in item 5.1 above, have the same properties as the $\varphi_n$'s in definition above. Replacing each $\varphi_n$ by $(\varphi_n +\varphi^*_n)/2$, we can suppose, without loss of generality, that the $\varphi_n$ are symmetric with respect to $\tau$ so that the corresponding $T_n$ are completely positive, self-adjoint contractions on $L^2(N,\tau)$.
\par\noindent
Let $\{\xi_k\in L^2(N,\tau):k\in \mathbb{N}\}$ be an orthonormal basis for the left $B$-module $L^2(N,\tau)$. Recall that this means that $L^2(N,\tau)=\overline{\oplus_{k\in\mathbb{N}}B\xi_k}$.
\par\noindent
For any $k\in\mathbb{N}$ one has $\lim_{n\to +\infty} T_n\xi_k =\xi_k$ and hence there exists $n_k\in\mathbb{N}$ such that
\[
(\xi_j|(I-T_{n_k})\xi_j)\le 2^{-k}\, ,\qquad k\in\mathbb{N}\, ,\quad j\in \{0,1,\dots ,k\}\, .
\]
Let us consider the quadratic form $(\E,\F)$ on $L^2(n,\tau)$ defined by
\[
\E[\xi]:=\sum_{k\in\mathbb{N}}(\xi|(I-T_{n_k})\xi)\qquad \xi\in L^2(N,\tau)\, .
\]
The domain $\F\subseteq L^2(N,\tau)$ being understood as the subspace where the quadratic form is finite. Note first that $\E$ is densely defined since $b\xi_j\in\F$ for all $b\in B$ and $j\in\mathbb{N}$. Noticing that each $\xi\mapsto (\xi|(I-T_{n_k})\xi)$ is a bounded symmetric Dirichlet form on $L^2(N,\tau)$ we see that $(\E,\F)$ is a lower semicontinuous, hence closed Dirichlet form on $L^2(N,\tau)$. The $B$-modularity of $(\E,\F)$ being obvious, what is left to prove is the relative discrete spectrum property.
\par\noindent
The generator $(L,D(L))$ associated to $(\E,\F)$, given by $L=\sum_{k\in\mathbb{N}}(I-T_{n_k})$,
appears as the increasing limit $L=\lim_{m\uparrow +\infty} L_m$ of the bounded operators
\[
L_m :=\sum_{k=0}^m (I-T_{n_k})=(m+1)I-\Theta_m
\]
where $\Theta_m:=\sum_{k=0}^m T_{n_k}$.The important fact is that $\Theta_m$ belongs to the compact ideal space $\mathcal{J}(\langle N,B\rangle)$.
\par\noindent
Let $q_m$ be the spectral projection of $\Theta_m$ corresponding to the interval $[0,(m+1)/2]$ and $p_m:=I-q_m$ the spectral projection corresponding to the interval $((m+1)/2,m+1]$. On one hand we have $\tau (p_m)<+\infty$, which implies $(I+L)^{-1/2}p_m \in \mathcal J(\langle N,B \rangle)$. On the other hand by spectral calculus we have
\[
q_m(I+L)^{-1}q_m \le q_m(I+L_m)^{-1}q_m\le \frac{2}{m+1}q_m
\]
which implies
\[
\|q_m(I+L)^{-1/2}\|\le \sqrt{\frac{2}{m+1}}\, .
\]
Finally, we have $(I+L)^{-1/2}=\lim_{m\to +\infty} (I+L)^{-1/2}p_m$  for the uniform norm, which implies that $(I+L)^{-1/2}$ and $(I+L)^{-1}$ are in $\mathcal{J}(\langle N,B\rangle)$.
\end{proof}
\begin{rem}
It would be very interesting to generalize the above result to type III (factors) inclusions. The first difficulties rely on a convenient identification of the ideal of relative compact operators.
\end{rem}

\section{Relative Property (H) for inclusions of discrete groups and conditionally negative definite functions}

In this section we extend a well known characterization of groups with Property (H) in terms of the existence of a proper c.n.d. function, to inclusions of discrete groups $H< G$.
\par\noindent
We denote by $\lambda$ and $\rho$ the left and right regular representation of $G$ in the Hilbert space $l^2(G)$. If $(\delta_t)_{t\in G}$ is the canonical orthonormal basis of $l^2(G)$, then
$$\lambda(s)\delta_t=\delta_{st} \qquad \rho(s)\delta_t=\delta_{ts^{-1}} \qquad s,t\in G\,.$$
The associated inclusion of von Neumann algebras is $$B=L(H)=\lambda(H)''\subset \lambda(G)'' = L(G)=N\,.$$
$L(G)$ has the canonical finite trace $\tau:L(G)\to\mathbb{C}$ determined by $\tau(\lambda(s))=\delta_e(s)$ for $s\in G$. The standard space $L^2(L(G),\tau)$ identifies canonically with $l^2(G)$ through the unitary map determined by $\lambda(s)\xi_\tau\to\delta_s$ for $s\in G$.  The projection $e_B$ of the basic construction is the orthogonal projection from $l^2(G)$ onto its subspace $l^2(H)$ given by the multiplication operator by the characteristic function $\chi_H$ of the subset $H$ of $G$.

\begin{lem}\label{basicforgroups}{(Basic construction for group inclusions).}

\item i) The basic construction $\big<N,B\big>$ is the commutant $\rho(H)'$ of the right regular representation restricted to $H$.

\item ii) For $T\in \big<N,B\big>$, the map $\varphi_T\,: G\to \mathbb{C}$ given by $s \mapsto (\delta_s,T\delta_s)$ is right $H$-invariant\,: $$\varphi_T(sh)=\varphi_T(s)\qquad s\in G\,,\;h\in H\,.$$

\item iii) The canonical trace on $\big<N,B\big>$ is given by the formula
\[
Tr(T)=\sum_{\widetilde s\in G/H}\varphi_T(s)=\sum_{\widetilde s\in G/H}(\delta_s,T\delta_s) \qquad T\in  \big<N,B\big>\]
where it is understood that the map $\widetilde s\mapsto s$ indicates a section of the projection $G\to G/H$.
\end{lem}

\begin{proof} i) On one side, we have $\big<N,B\big>=(JBJ)'$. On the other side, it is an elementary fact that $J\lambda (h)J=\rho(h)$, $h\in H$.
ii) For $h\in H$ and $T\in \mathcal{B}(l^2(G))$ commuting with $\rho(h)$, we compute
$$\varphi_T(sh)=(\rho(h)^{-1}\delta_s,T\rho(h)^{-1}\delta_s)=(\delta_s,\rho(h)T\rho(h)^{-1}\delta_s)=\varphi_T(s)\qquad s\in G,\, h\in H.$$
iii) Fix an arbitrary section $\sigma\,: \, G/H\to G$ of the homogeneous space $G/H$. The  formula
$$\varphi(T)=\sum_{\gamma\in G/H}(\delta_{\sigma(\gamma)},T\delta_{\sigma(\gamma)})\in [0,+\infty]\qquad T \in \big<N,B\big>_+=\rho(H)'_+$$
defines a normal faithful weight $\varphi$ on the von Neumann algebra $\big<N,B\big>$. We claim that this weight is equal to the natural trace $Tr$.
\par\noindent
For $x\in L(G)$ with Fourier expansion $x=\sum_{s\in G}x(s)\lambda(s)$ (the series $\{x(s)\}_{s\in G}$ converges at least in the $\ell^2$-sense), let us compute
\begin{equation*}\begin{split}
\varphi(x^*e_Bx)&=\sum_{\gamma\in G/H}( \delta_{\sigma(\gamma)},x^*e_Bx\delta_{\sigma(\gamma)})
\\ &=\sum_{\gamma\in G/H}||e_Bx\delta_{\sigma(\gamma)}||^2
\end{split}\end{equation*}
with
\begin{equation*}\begin{split}
||e_Bx\delta_{\sigma(\gamma)}||^2&=||e_B\sum_{s\in G} x(s)\delta_{s\sigma(\gamma)}||^2\\
&=||e_B\sum_{s\in G} x(s\sigma(\gamma)^{-1})\delta_{s}||^2 \\
&=||\sum_{h\in H} x(h\sigma(\gamma)^{-1})\delta_h||^2 \\
&=\sum_{s\in H\sigma(\gamma)^{-1}}||x(s)||^2
\end{split}\end{equation*}
and finally
\begin{equation*}\begin{split}
\varphi(x^*e_Bx)=\sum_{\gamma\in G/H} \sum_{s\in H\sigma(\gamma)^{-1}}||x(s)||^2 = \sum_{s\in G}||x(s)||^2=\tau(x^*x)=Tr(x^*e_Bx)\,.
\end{split}\end{equation*}
This proves that the normal weight $\varphi$ is semifinite and that it coincides with the canonical trace $Tr$ on elements of the form $x^*e_B x$, $x\in N$. \par\noindent
The Radon-Nykodim derivative $d\varphi/dTr$ must be equal to $1$, hence $\varphi=Tr$.

\end{proof}

\begin{thm}\label{propHforgroups}
Let $G$ be a countable discrete group and $H<G$ a subgroup. Then the inclusion of von Neumann algebras $L(H)\subset L(G)$ has the relative Property (H) if and only if there exists a conditionally negative type function $\ell:G\to [0,+\infty)$ such that
\begin{itemize}
  \item [i)] $\ell|_H=0$
  \item [ii)] $\ell$ is proper on $G/H$.
\end{itemize}
\end{thm}
\begin{proof}
Let us suppose that i) and ii) are satisfied. Then for any $t>0$, $e^{-t\ell}$ is a positive type function equal to $1$ on $H$ and it induces a normal, completely positive, trace preserving contraction $\phi_t$ on the von Neumann algebra $L(G)$, characterized by
\[
\phi_t (\lambda_u)=e^{-t\ell(u)}\lambda_u\qquad u\in G\, .
\]
Let $(\pi, H_\pi)$ be the orthogonal representation and $c:G\to H_\pi$ the 1-cocycle associated to the c.n.d. function $\ell$. If $\ell (u)=0$ then $0=\ell (u)=\|c(u)\|^2_{H_\pi}$ so that $c(u)=0$ for all $u\in H$. Then, $c(vu)=c(v)+\pi (v)c(u)=c(v)$ for all $v\in G$ and $u\in H$ and $\ell(vu)=\ell(v)$, $v\in G$, $u\in H$.
Finally, as $\ell (v^{-1})=\ell (v)$ for all $v\in G$, one also has $\ell (uv)=\ell(v)$ for all $v\in G$ and $u\in H$.
\par\noindent
Consequently we shall have
\[
\phi_t (\lambda_v\lambda_u)=e^{-t\ell (vu)}\lambda_v\lambda_u =e^{-t\ell (v)}\lambda_v\lambda_u =\phi_t (\lambda_v)\lambda_u\qquad v\in G\, ,\,\,u\in H
\]
and
\[
\phi_t (ab)=\phi_t(a)b\qquad a\in L(G)\, ,\,\, b\in L(H).
\]
Similarly
\[
\phi_t (ba)=b\phi_t(a)\qquad a\in L(G)\, ,\,\, b\in L(H).
\]
This proves that $\phi_t$ is a $L(H)$-bimodular map for all $t>0$.
\par\noindent
The self-adjoint operator $T_t$ on $l^2(G)$, induced by $\phi_t$ on $L(G)$, is just the multiplication operator by the function $e^{-t\ell}$. Its spectrum coincides with the set of values $e^{-t\lambda}$ where $\lambda$ runs in the range $\ell (G)$ which is, by assumption ii), a discrete subset of $[0,+\infty)$.
\par\noindent
The eigenspace $E_\lambda$, corresponding to the eigenvalue $e^{-t\lambda}$, is the set of functions in $l^2(G)$ supported by $S_\lambda :=\{s\in G:\ell(s)=\lambda\}$. Again by assumption ii), $S_\lambda$ is a finite union of right $H$-cosets: $S_\lambda =\bigcup_{i=1}^k u_i H$ for some $u_1 ,\cdots , u_k\in G$. The corresponding spectral projection $P_\lambda$ is then the sum of projections on those cosets:
$$P_\lambda=\sum_{i=1}^k \text{ multiplication by }\chi_{u_iH}\,.$$
\par\noindent
Since, for $B:=L(H)\subset L(G)=N$, the projection $e_B$ is just the multiplication operator by the characteristic function $\chi_H$ of the subgroup, the projection $\lambda_u e_B\lambda_u^{-1}$, for $u\in G$, is the multiplication operator by the function $\chi_{uH}$. The spectral projection $P_\lambda$ onto $E_\lambda$ is then a finite sum of projections of trace one (for the trace $Tr$), and its  trace is equal to the number of right $H$-cosets in $S_\lambda$\,:
\[
{\rm Tr\,}(P_\lambda)={\rm Tr\,}(\sum_{i=1}^k \lambda_{u_i} e_B\lambda_{u_i}^{-1})=\sum_{i=1}^k {\rm Tr\,}(\lambda_{u_i} e_B\lambda_{u_i}^{-1})=\sum_{i=1}^k {\rm Tr\,}(e_B)=\sum_{i=1}^k 1 =k\, .
\]
This proves that $T_t$ belong to $\mathcal J\big(\big<N,B\big>\big)$, which ends the proof in the forward direction.

\medskip\noindent
Conversely, let us suppose that the inclusion $L(H)\subset L(G)$ has the relative property $(H)$. By Theorem \ref{propH}, there exists an $L(H)$-invariant symmetric Dirichlet form with generator $L$ and discrete spectrum relative to the subalgebra $B=L(H)$.
\par\noindent
Let us observe that, for $\varepsilon >0$, the resolvent maps $(I+\varepsilon L)^{-1}$ are a completely positive, normal contractions of the von Neumann algebra $N=L(G)$ and that the map
$$\omega_\varepsilon\,:\,G\to [0,+\infty)\qquad\; \omega_\varepsilon(s)=(\delta_s,(I+\varepsilon L)^{-1}\delta_s)=\tau(\lambda(s)^*(I+\varepsilon L)^{-1}(\lambda(s)))$$
is positive definite on $G$ and $H$-right invariant. Notice that $\o_\varepsilon$ is symmetric because it is positive definite and real, as $L$ is self-adjoint. Moreover, by the weak$^*$-continuity of the resolvent, one has
$$\lim_{\varepsilon\to 0} \omega_\varepsilon(s)=1\qquad s\in G.$$
We claim that $\omega_\varepsilon$ vanishes at infinity on the quotient space $G/H$. This holds true because $(I+\varepsilon L)^{-1}$ being in the compact ideal $\mathcal J\big(\big<N,B\big>\big)$, it will be a uniform limit of trace class operators $T_n\in L^1(\big<N,B\big>,Tr)$ and that, for such $T_n$, the function $s\mapsto (\delta_s,T_n \delta_s)$ is summable on $G/H$ (by Lemma \ref{basicforgroups} iii)) thus vanishing at infinity.
\par\noindent
Let $(F_k)_{k\geq 1}$ be an increasing family of finite subsets of $G$ such that $\cup_k F_k=G$. For any $k$, let us choose $\varepsilon_k>0$ such that
\[
0\leq 1-\omega_{\varepsilon_k}(s)\leq 2^{-k}\qquad s\in F_k
\]
and consider the function $\ell\,:\,G\to [0,+\infty)$ defined by
\[
\ell(s)=\sum_{k=1}^\infty (1-\omega_{\varepsilon_k}(s))\qquad s\in G.
\]
By the choice of $\varepsilon_k$, the series converges for any $s$ in any $F_k$ and thus for any $s\in G$ so that $\ell$ is well defined on $G$. The function $\ell$ is c.n.d. as a sum of c.n.d. functions and a right $H$-invariant function as a sum of right-$H$-invariant functions. It can thus be considered as a function on $G/H$.
\par\noindent
Moreover, as the functions $\omega_\varepsilon$ vanish at infinity on $G/H$, the set $\Gamma_k=\{\widetilde s\in G/H\,|\,\omega_{\varepsilon_k}(s)\geq 1/2\}$ is finite and for $\widetilde s\not \in \cup_{k=1}^N \Gamma_k$, one has $\ell(s)\geq N/2$. Hence, for any $N\in \mathbb{N}$, the set $\{\widetilde s\in G/H\,|\, \ell(s)\leq N/2\}$ is finite, which proves that $\ell$ is proper on $G/H$.
\end{proof}
\begin{cor} If $H<G$ and $L(H)\subset L(G)$ has relative property $(H)$, then $H$ is a quasi normal subgroup: each orbit of the left action of $H$ on the right-cosets space $G/H$ is finite.
\end{cor}
\begin{proof} The c.n.d. function $\ell$ constructed in Theorem \ref{propHforgroups} is left $H$-invariant, as it clearly satisfies $\ell(s^{-1})=\ell(s)$ for $s\in G$. It is thus constant on the orbits of the left action of $H$ on the right cosets space $G/H$. A subset of $G/H$ on which the proper function $\ell$ is constant must be finite.
\end{proof}

\begin{cor}
If $H<G$ is normal, then the inclusion $L(H)\subset L(G)$ has relative property $(H)$ if and only if the quotient group $G/H$ has the Haagerup property.
\end{cor}

\begin{ex}
Let us consider the free group $\FF_2$ with generators $a,b\in \FF_2$ and the abelian subgroup $H:=\{a^k:k\in \Z\}$ isomorphic to the additive group $\Z$ of integer numbers.
\par\noindent
$H$ is not quasi-normal in $\FF_2$, so that the inclusion $H<G$ has not the relative property $(H)$, though both $G$ and $H$ have property $(H)$.
\par\noindent
More generally, for countable discrete groups $G_1,G_2$, the inclusion $G_1<G_1*G_2$ is not quasi-normal as soon as $G_1$ is infinite and $G_2$ has at least two elements. Hence, this inclusion has not the relative property $H$.
\end{ex}

\end{document}